\documentclass[a4paper, 11 pt, twoside]{amsart}
\pdfoutput=1

\usepackage{mathtools}
\mathtoolsset{showonlyrefs=true}
\usepackage{amsmath,amssymb,amsthm,amsfonts, mathrsfs, fancyhdr}
\usepackage[usenames,dvipsnames,usenames,dvipsnames,svgnames,table]{xcolor}
\usepackage[expansion=false]{microtype}

\usepackage{enumitem}
\makeatletter
\newcommand{\mylabel}[2]{#2\def\@currentlabel{#2}\label{#1}}
\makeatother

\setlist[description]{leftmargin=*}
\usepackage{amscd}
\usepackage[active]{srcltx}
\usepackage{verbatim}
\usepackage{epsfig,graphicx,color}
\usepackage{graphicx}
\usepackage{mathtools,xparse}
\usepackage[english]{babel}

\usepackage{pgf,tikz}
\usetikzlibrary{arrows}
\usetikzlibrary[patterns]
\usepackage{capt-of}
\usepackage{float}

\definecolor{qqwwtt}{rgb}{0,0.4,0.2}\definecolor{ttzzqq}{rgb}{0, 0.66, 0.47} \definecolor{wrwrwr}{rgb}{0.3803921568627451,0.3803921568627451,0.3803921568627451}\definecolor{ffqqqq}{rgb}{1,0,0}

\usepackage{palatino,hhline, bbm}
\usepackage[left=2.7cm,right=2.7cm,top=3cm,bottom=3cm]{geometry}

\usepackage[unicode,bookmarks, pdftex, backref = page]{hyperref}
\definecolor{newblue}{RGB}{0,146,183}
\hypersetup{colorlinks=true,citecolor=newblue,linkcolor=newblue, urlcolor=Orange,
pdfpagemode=UseNone, breaklinks=true}

\newtheorem{theorem}{Theorem}[section]

\newtheorem{lemma}[theorem]{Lemma}
\newtheorem{corollary}[theorem]{Corollary}

\theoremstyle{definition}
\newtheorem{remark}[theorem]{Remark}
\newtheorem{example}[theorem]{Example}
\newtheorem*{assumption}{Main Assumption}

\newcommand{\ben}{\begin{enumerate}}
\newcommand{\een}{\end{enumerate}}

\newcommand{\e }{\varepsilon }

\newcommand{\be}{\begin{equation}}
\newcommand{\ee}{\end{equation}}

\makeatletter
\let\orgdescriptionlabel\descriptionlabel
\renewcommand*{\descriptionlabel}[1]{%
  \let\orglabel\label
  \let\label\@gobble
  \phantomsection
  \edef\@currentlabel{#1}%
  \let\label\orglabel
  \orgdescriptionlabel{#1}%
} \makeatother

\DeclarePairedDelimiter{\norm}{\lVert}{\rVert}
\NewDocumentCommand{\normL}{ s O{} m }{%
  \IfBooleanTF{#1}{\norm*{#3}}{\norm[#2]{#3}}_{L_2(\Omega)}%
}

\title[A weak comparison principle in tubular neighbourhoods]{A weak
comparison principle in tubular neighbourhoods of embedded
manifolds}

\author[Francesco Polizzi]{Francesco Polizzi}
\address{Dipartimento di Matematica e Informatica
\newline\indent
Universit\`a della Calabria
\newline\indent
Ponte Pietro Bucci 31B, I-87036 Arcavacata di Rende, Cosenza, Italy}
\email{polizzi@mat.unical.it}

\author[Pietro Sabatino]{Pietro Sabatino}
\address{Dipartimento di Matematica
\newline\indent
Universit\`a di Roma Tor Vergata
\newline\indent
Via della Ricerca Scientifica - 00133 Roma}
\email{pietrsabat@gmail.com}

\author[Berardino Sciunzi]{Berardino Sciunzi}
\address{Dipartimento di Matematica e Informatica
\newline\indent
Universit\`a della Calabria
\newline\indent
Ponte Pietro Bucci 31B, I-87036 Arcavacata di Rende, Cosenza, Italy}
\email{sciunzi@mat.unical.it}


\begin{document}
\begin{abstract}
We study weak solutions to degenerate quasilinear elliptic equations,
involving first order terms, in unbounded tubular domains. In
particular we show that, under suitable hypotheses, the
weak comparison principle holds if the domain  is narrow enough.
\end{abstract}

\maketitle

%\tableofcontents

%\textcolor{red}{[TO DO]} Explain the notation $a(x, y)$ in order to indicate
%the pullback of $a(z)$ and so on.

%\textcolor{red}{[TO DO]} Find the relation between $\nabla_g$ and
%$\nabla_{\mathbb{R}^n}$, and use it in the extimates \eqref{eq:A2r-final},
%\eqref{eq:B2r-final}, \eqref{eq:C2r-final}. The best thing would be, in my
%opinion, to put a Lemma (or a Corollary) after Remark \ref{rmk:dvol-h}.

%\textcolor{red}{[TO DO?]} Give examples (both compact and non-compact)
%where our conditions \eqref{eq:h1}, $\ldots$, \eqref{eq:h4} are satisfied.

%\textcolor{red}{[TO DO Pietro e Francesco]}
%From Subsection \ref{subsection:testfunctions},
%clarify the statement: ``Via density arguments, it is standard to see
%that\dots''.

\section{Introduction and statement of the  main results}\label{introdue}
We want to the study the \emph{weak comparison principle} for
degenerate quasilinear elliptic equations with first order terms.
More precisely, we consider $u,v\in C^1(\Omega)\cap
C^0(\overline\Omega)\cap L^\infty (\Omega)$, weak solutions to the
problem
\begin{equation}\label{Eq:WCP}
\begin{cases}
-{\hbox {\rm div}} (a(z) \nabla u )+\Lambda |\nabla u|^q\leq f(z, \,
u) & \text{
in }\Omega\\
-{\hbox {\rm div}} (a(z) \nabla v )+\Lambda |\nabla v|^q\geq f(z, \,
v) & \text{
in }\Omega,\\
\end{cases}
 \end{equation}
 where  $q\geq 1$, $\Omega\subseteq \mathbb{R}^n$ and  denoting with $z$ a point in $\mathbb{R}^n$. Equivalently for smooth solutions, by using the divergence theorem, \eqref{Eq:WCP} can be rephrased as
\begin{equation}\label{weak1}
\int_{\Omega}\,a(z)\nabla u \cdot \nabla\psi+\Lambda |\nabla
u|^q\,\psi\,dz\,\leq\,\int_{\Omega}f(z, \, u)\,\psi\,dz
\end{equation}
\begin{equation}\label{weak2}
\int_{\Omega}\,a(z)\nabla v \cdot \nabla\psi+\Lambda |\nabla
v|^q\,\psi\,dz\,\geq\,\int_{\Omega}f(z, \, v)\,\psi\,dz,
\end{equation}
for every test-function $\psi\in C^\infty_c(\Omega) $ with $\psi\geq
0$. In all the paper  $a(z)$ is a nonnegative weight satisfying Assumption  $\mathrm{\ref{eq:h1}}$. In particular, since $a(z)$
may vanish,  the operator could be degenerate at some points, causing that solutions could be not of class $C^2$. In such a case \eqref{weak1} and \eqref{weak2} represent the right way of understanding our problem and in fact all our proofs are carried out exploiting only the weak formulation.\\

As customary in the literature, we say that the  \emph{weak
comparison principle} for problem \eqref{Eq:WCP} holds in $\Omega$
if, imposing that $u\leq v$ on $\partial\Omega$, then it follows
that $u\leq v$ in the whole of $\Omega$. A leading example is
provided by the \emph{weak maximum principle}, namely by the case
when  $v=0$. It is well-known that the maximum principle does
not
hold in general domains, but it holds if the domain is sufficiently small or after imposing suitable a-priori estimates on the solutions. For weak solutions this can be seen e.g. as a consequence of the  Poincar\'{e} inequality; on the other hand, if one  considers smooth solutions then direct methods are also available, see \cite{GT83}. \\

In the present work we deal with the case of unbounded domains. Here
the situation is quite well understood for domains with flat
boundary and in particular for narrow strips. In fact, maximum and
comparison principles in strips have many applications to the study
of symmetry and monotonicity properties of the solutions, see for
instance \cite{BCN1,BCN2,BCN3,BCN5} and especially \cite{BNV}.
 The case of domains with a possible different geometry has been considered in \cite{cabre}, where deep results have been obtained by means of Alexandroff-Bakelmann-Pucci-type estimates. \\

In our analysis, we suppose that the (possibly unbounded and) smooth
domain $\Omega\subset \mathbb{R}^{n}$ is contained in a normal tube
$\mathscr{B}(M, \, \e)$, where $M \subset \mathbb{R}^n$ is a
smoothly embedded submanifold. Our crucial idea is to manipulate the
inequalities \eqref{weak1} and \eqref{weak2} by using the so-called
\emph{Fermi coordinates} associated with  the normal tube
$\mathscr{B}(M, \, \e)$, see \cite[Chapter 2]{G04}, \cite[Section
3]{PPS14}. The weak comparison principle then follows by showing
that $\widetilde{\mathcal{
    L}}_{R}$=0 for all $R >0$, where
\begin{equation*}
\widetilde{\mathcal{
    L}}_{R}\,:=\int_{\Omega_{2R}}\,a(z) \,(u-v)^+ \, |\nabla(u-v)|^2\, \mathrm{dvol}_g,
\end{equation*}
 see Section \ref{section:proof} for notation and details. The desired vanishing is established by proving that $\widetilde{\mathcal{L}}_{R}$ can be estimated in terms of $\widetilde{\mathcal{L}}_{2R}$,  see Equation \eqref{eq:theta-final}, allowing us to exploit the simple but useful Lemma \eqref{Le:L(R)}.\\

This kind of technique seems to be very flexible.  In particular,
having in mind the fact that the technique is manly based on the
weighted Sobolev inequality, we believe that it could be also used
when dealing with $p$-Laplace
inequalities as well as with fully nonlinear problems. \\

Let us now introduce some notation and terminology that will be used
in the sequel. Let $M$ be a smooth submanifold of $\mathbb{R}^n$ of
dimension $n-k$, and let $i \colon M \hookrightarrow \mathbb{R}^n$
be the embedding map. By a \emph{normal tube} of constant radius
$\varepsilon > 0$ centred at $M$ we mean a tubular neighbourhood of
$M$ given by a \emph{disjoint} union
\begin{equation*}
  \mathscr{B}(M, \, \e)=\bigsqcup_{p\in M}B(p, \, \e),
\end{equation*}
where $B(p, \, \e)$ is a $k$-dimensional ball of radius $\e$ centred
at $p \in M$ and contained in the normal subspace $N_p M \subset
\mathbb{R}^n$. Moreover, given any subset $A \subseteq M$ we define
\begin{equation} \label{eq:B_A}
  \mathscr{B}(A, \, \e) = \bigsqcup_{p\in A}B(p, \, \e) \subseteq
  \mathscr{B}(M, \, \e).
\end{equation}
We are now ready to make our basic hypothesis on the domain
$\Omega$.
\begin{assumption} %\label{ass:tubular}
The submanifold $M$ admits a normal tube $\mathscr{B}(M, \, \e)$ for
some $\e >0$ and we have $\Omega\subseteq \mathscr{B}(M, \, \e)$.
\end{assumption}

\begin{remark} \label{rmk:tubular}
Let us consider the \emph{normal exponential map}
\begin{equation*}
\mathrm{exp}^{\perp} \colon NM \longrightarrow \mathbb{R}^n,
\end{equation*}
sending the pair $(p, \, v)$, with $p \in M$ and $v \in N_pM$, to
$p+v$ (here we are identifying $T_p \mathbb{R}^n$ with
$\mathbb{R}^n$ in the standard way). We can define the \emph{normal
injectivity radius} of $M$ as the strictly positive function $\rho
\colon M \to \mathbb{R}$ that associates to $p \in M$ the supremum
$\rho(p)$ of all $\delta \leq 1$ such that the restriction of
$\mathrm{exp}^{\perp}$ to the neighbourhood of the zero section
\begin{equation*}
V_{\delta}(p):=\{(p', \, v') \in NM \; : \; |p-p'| < \delta, \, |v'|
< \delta \}
\end{equation*}
is a diffeomorphism onto its image in $\mathbb{R}^n$. Therefore $M$
admits a normal tube $\mathscr{B}(M, \, \e)$ of radius $\e$ if and
only if $\rho(p) \geq \e$ for all $p \in M$. In fact, in this case
we can identify $\mathscr{B}(M, \, \e)$ with the open neighbourhood
of $M$ in $\mathbb{R}^n$ defined by
\begin{equation*}
\{x \in \mathbb{R}^n \; : \; d(x, \, M) < \e\},
\end{equation*}
see \cite[Proposition 7.26]{ONeill83} and \cite[Theorem
10.89]{Lee03} for more details.

In particular, if $M$ is compact and $\e$ is sufficiently small then
$\mathscr{B}(M, \, \e)$ always exists, cf.  \cite[Proposition
3.7.18]{Con01} and \cite[Section 3]{PPS14}. On the other hand, there
are also plenty of non-compact submanifolds of $\mathbb{R}^n$
admitting normal tubes of constant radius, for instance all the
linear subspaces have this property.

It is important to observe that it is not possible to give
sufficient conditions for the existence of $\mathscr{B}(M, \, \e)$
in terms of curvature bounds for $M$. To see this, consider the
smooth curve $M \subset \mathbb{R}^3$  parametrized by $x \mapsto
(\sin x, \, \cos x, \, \arctan x)$, namely a spiral with the
distance between the coils tending to $0$. A straightforward
computation shows that its curvature $\kappa(x)$ satisfies
$\frac{1}{2} \leq \kappa(x) < 1$, hence it is is globally bounded,
but $M$ admits no tubular neighbourhood of constant radius.
\end{remark}
\bigskip
In addition to our Main Assumption, will also need Assumptions
$\mathrm{\ref{eq:h1}}$, $\mathrm{\ref{eq:h2}}$,
$\mathrm{\ref{eq:h3}}$, $\mathrm{\ref{eq:h4}}$ stated below.
\begin{description}
\item[Assumption \mylabel{eq:h1}{A1}] For all $p \in M$, write
\begin{equation*}
  \Omega_{p}\,:= \Omega \cap B(p, \, \e).
\end{equation*}
We require that the weight $a$ in \eqref{Eq:WCP} is a non-negative function such that $a\in
C^0(\overline\Omega)\cap L^\infty(\Omega)$ and we suppose that there
exists a constant $C_a$ such that
\begin{equation}
  \int_{\Omega_{p} }
  (a_{|\Omega_p})^{-t}\,dy^1\dots dy^{k}
    \leq
  C_a\quad \text{for some }\,t>k, \, \text{uniformly with respect to }\,p.
  \end{equation}
Here $a_{|\Omega_p}$ denotes the restriction of $a$ to $\Omega_p$
and $y^1,\dots, y^k$ are the Euclidean coordinates in the
$k$-dimensional linear subspace $N_pM$ of $\mathbb{R}^n$ containing
$\Omega_p$. We observe that this kind of condition is quite
well-known in the context of
weighted Sobolev spaces since of the works \cite{murthy,Tru}. \\

\item[Assumption \mylabel{eq:h2}{A2}] $f(z,\, \cdot)$ is a continuous function, uniformly
Lipschitz with respect to $z$. In other words, for every $m>0$ there
is a positive
 constant
$L_f=L_f(m)$ such that for every $z \in \Omega$ and every $ u, \, v
\in [-m, \, m]$ we have
\begin{equation}
  \norm{f(z, \, u) - f(z, \, v)} \le L_f | u-v |.
\end{equation}
As a leading example we may consider the case
\begin{equation*}
f(z,\, u)=h(z)g(u),
\end{equation*}
where $h\in L^\infty (\Omega)$ and $g$ is a locally Lipschitz,
 continuous function. \\

\item[Assumption \mylabel{eq:h3}{A3}]
We consider an oriented  atlas of $M$
\begin{equation}
  \label{eq:atlas-Gamma}
  \{(U_\alpha, \, \phi_\alpha)\}, \quad M =
\bigcup_{\alpha}U_{\alpha}, \quad \phi_{\alpha} \colon U_{\alpha}
\stackrel{\simeq}{\longrightarrow} V_{\alpha} \subset
\mathbb{R}^{n-k}
\end{equation}
such that the functions $i_{\alpha} \circ \phi^{-1}_\alpha \colon
V_{\alpha} \longrightarrow \mathbb{R}^n$ (where $i_{\alpha} \colon
U_{\alpha} \hookrightarrow \mathbb{R}^n$ denotes the restriction of
$i \colon  M \hookrightarrow \mathbb{R}^n$ to $U_{\alpha}$),
together with their first and second derivatives, are uniformly
bounded with respect to $\alpha$.

\begin{remark} An atlas as above exists for \emph{every} smooth manifold $M$. In fact, let us start with an atlas $\{(U_\alpha, \, \phi_\alpha')\}$ given by a countable basis of precompact smooth coordinate balls (it exists by \cite[Lemma 1.11]{Lee03}). On each chart $U_{\alpha}$, the functions $i_{\alpha} \circ (\phi_\alpha')^{-1}$ and their derivatives are bounded. We now set $\phi_{\alpha}:=A_{\alpha}\phi_\alpha'$, where $A_{\alpha}$ is a positive constant; then, provided that we choose the $A_{\alpha}$ large enough, we can make the bounds above uniform on $\alpha$.   \\
\end{remark}

\item[Assumption \mylabel{eq:h4}{A4}]
There exist a point $\bar p\in M$ and constants $C_1, \gamma, \, R_0
> 0$ such that
\begin{equation} \label{eq:ball}
  \mathrm{vol} \big( B_{M}(\bar p, \, R) \big)\le C_1 R^{\gamma}
\end{equation}
for all $R > R_0$. Here we define
\begin{equation*}
  B_{M}(\bar p, \, R) := \{ p \in M \ | \ d_M(\bar p, \, p) < R \},
\end{equation*}
where $d_M$ denotes the distance function induced on $M$ by the
pull-back of the Euclidean metric of $\mathbb{R}^n$, and
$\mathrm{vol} \big( B_{M}(\bar p, \, R) \big)$ stands for the
measure of $B_{M}(\bar p, \, R)$, computed again with respect to the
Euclidean metric on $M$.
\begin{remark} \label{rem:geodesic-ball}
It is known that estimate \eqref{eq:ball} holds (at every point
$\bar{p} \in M$) if $M$ is a complete Riemannian manifold with
non-negative Ricci curvature (since the volume is an intrinsic
geometric invariant, the fact that $M$ is a submanifold of
$\mathbb{R}^n$ is irrelevant here). Indeed, in this case, the
Bishop-Gromov inequality gives
\begin{equation*}
  \mathrm{vol} \big( B_{M}(\bar p, \, R) \big)\leq \Gamma_{\dim(M)},
\end{equation*}
where $\Gamma_{\dim(M)}$ stands for the volume of the Euclidean ball
of $\mathbb{R}^{\dim (M)}$, see \cite[Lemma 1.5 p. 247]{Pet98}.

On the other hand, without the non-negativity assumption on the
Ricci curvature, \eqref{eq:ball} in general does not hold. For
instance, the volume of a ball of radius $R$ in a symmetric space
$M$ of non-compact type grows exponentially with $R$, cf. \cite[p.
209]{HTT3}. In order to have an embedded example of the last
situation in any dimension $m$, recall that the hyperbolic space
$\mathbb{H}^m$ admits a smooth, isometric embedding into
$\mathbb{R}^n$ with $n=6m-5$, see \cite{Bl55}.
\end{remark}
\end{description}
\bigskip

We are now ready to state our main result.
\begin{theorem}\label{th:wcpstrip}
  Let $\Omega\subseteq  \mathscr{B}(M, \, \e) \subseteq \mathbb{R}^n$
  be a smooth domain
and let $u, \, v\in C^1(\Omega)\cap C^0(\overline\Omega)\cap
W^{1,\infty} (\Omega)$ be weak solutions to \eqref{Eq:WCP}, with
\begin{equation*}
u\leq v \quad \text{on}\; \; \partial\Omega\,.
\end{equation*}
Assume moreover that Assumptions $\mathrm{\ref{eq:h1}}$,
$\mathrm{\ref{eq:h2}}$, $\mathrm{\ref{eq:h3}}$ and
$\mathrm{\ref{eq:h4}}$ are satisfied. Then there exists a
positive constant
\begin{equation*}
\e_0\,:=\,\e_0(\|u\|_{W^{1,\infty}}, \, \|v\|_{W^{1,\infty}}, \, t,
\, a, \, q, \, n, \, k, \, f, \Lambda, \, \gamma)
\end{equation*}
such that, if $0<\e<\e_0$, then
\begin{equation*}
u\leq v\quad\text{in}\; \; \Omega\,.
\end{equation*}
\end{theorem}
Apart from this Introduction, the paper contains two more sections.
More precisely, in Section \ref{section2} we collect some
preliminary results that are needed in what follows, whereas Section
\ref{section:proof} is devoted to the proof of Theorem
\ref{th:wcpstrip}.

\section{Preliminaries}\label{section2}

\subsection{A useful lemma}
We start by recalling a useful lemma already contained and exploited
in \cite{Fms2012}. For the reader's convenience we also provide the
proof.
\begin{lemma} \label{Le:L(R)}
Let $\theta >0$ and $\gamma>0$ be such that $\theta < 2^{-\gamma}$.
Moreover, let $R_0>0$, $C>0$ and
\begin{equation*}
\mathcal{L} \colon (R_0, \, + \infty) \rightarrow \mathbb{R}
\end{equation*}
be a non-negative and non-decreasing function such that
\begin{equation}\label{eq:L}
\begin{cases}
\mathcal{L}(R)\leq \theta \mathcal{L}(2R)+g(R) & \textrm{for all } R>R_0\\
\mathcal{L}(R)\leq CR^{\gamma} & \textrm{for all } R >R_0,
\end{cases}
\end{equation}
where $g \colon (R_0, +\infty)\rightarrow \mathbb{R}^+$ is such that
\begin{equation*}
\lim_{R\rightarrow +\infty}g(R)=0 .
\end{equation*}
Then
\begin{equation*}
\mathcal{L}(R)=0 \quad \textrm{for all } R > R_0.
\end{equation*}
\end{lemma}
\begin{proof}
Choose $\theta_1$ such that $ \theta <\theta_1<2^{- \gamma}.$  Then
there exists $ R_1=R_1(\theta_1) \geq R_0$ such that
\begin{equation*}
\theta \mathcal{L}(2R)+g(R) \leq \theta_1 \mathcal{L}(2R)  \quad
\textrm{for all } R \geq R_1,
\end{equation*}
and so, using the first inequality in \eqref{eq:L}, we get
\begin{equation}\label{eq:L(R)}
\mathcal{L}(R)\leq \theta_1 \mathcal{L}(2R) \quad \textrm{for all }
R \geq R_1.
\end{equation}
Iterating \eqref{eq:L(R)} and using the second inequality in
\eqref{eq:L} we obtain, for all $m \in \mathbb{N}^+$ and $R  \geq
R_1$,
\begin{align} \label{eq:L1}
0 \leq \mathcal{L}(R) & \leq \, \theta_1^{m} \mathcal{L}(2^{m}R)\\
 &\leq  \, C \theta_1^{m}(2^{m}R)^{\gamma} \\ 
&= \, C(2^{\gamma}\theta_1)^{m}R^{\gamma}.
\end{align}
Since $0 < 2^{\gamma} \theta_1 < 1$, by \eqref{eq:L1} we deduce
\begin{equation*}
0 \leq \mathcal{L}(R) \leq \lim _{m \rightarrow  +\infty}
C(2^{\gamma} \theta_1)^{m}R^\gamma=0  \quad \textrm{for all } R \geq
R_1,
\end{equation*}
that implies $\mathcal{L}(R) =0$ for all $R \geq R_1$. By Assumption
$\mathcal{L}$ is non-decreasing, so the claim follows.
\end{proof}

\subsection{Weighted Sobolev Spaces}
The  \emph{weighted Sobolev space}  with weight $\rho$ is defined as
the space $W^{1, \,2}(\Omega, \,\rho)$ of those functions in
$L^2(\Omega)$ having distributional derivative in $\Omega$ for which
the norm
\begin{equation}\label{normaaniso4}
\begin{split}
\|v\|_{W^{1,\,2}(\Omega, \,\rho)}=\|v\|_{L^2(\Omega)}+\|\nabla
v\|_{L^2(\Omega, \,\rho)} =\Big(\int_\Omega v^2\Big)^{\frac{1}{2}}+
\Big(\int_\Omega |\nabla v|^2\,\rho\Big)^{\frac{1}{2}}
\end{split}
\end{equation}
is bounded. The case $\rho=1$ corresponds to the classical Sobolev
space $W^{1, \, 2}(\Omega)$. We  also define the space $H^{1, \,
2}(\Omega, \, \rho)$ as the closure of $C^\infty (\Omega) \cap W^{1,
\, 2}(\Omega, \, \rho)$ and the space $H_0^{1, \, 2}(\Omega, \,
\rho)$ as the closure of $C_0^\infty (\Omega)$, with respect to the
norm \eqref{normaaniso4}. Let us now recall from \cite{murthy,Tru}
the following

\begin{theorem}\label{bvbdvvbidvldjbvlb}
Let $\mathcal{D} \subset \mathbb{R}^k$ be a smooth, bounded domain,
$\rho\in L^1(\mathcal{D})$ and $ C_\rho$ a constant such that
  \begin{equation*}
  \int_{\mathcal{D} } \frac{1}{\rho^t}\leq  C_\rho \qquad \text{for some }\,t>\frac{k}{2}.
  \end{equation*}
Then, for $1+\frac{1}{t}>\frac{2}{k}$, setting
\begin{equation*}
\frac{1}{ 2^*(t)}:=\frac{1}{2}-\frac{1}{k}+\frac{1}{2t}
\end{equation*}
it results $2<2^*(t)<+\infty$ and it follows that the space $H^{1,
\, 2}_{0}(\mathcal{D}, \, \rho)$ is continuously embedded in
$L^q(\mathcal{D})$ for $1\leqslant q\leqslant 2^*(t)$. More
precisely, there exists a constant $C_{t, \, k}$ such that
\begin{equation}\label{Sobolev}
\|w\|_{L^{2^*(t)}(\mathcal{D})}^2 \leqslant C_{t, \, k} \,
C_\rho^{-t} \|\nabla w\|_{L^2(\mathcal{D}, \,
\rho)}^2={C}_S\int_{\mathcal{D}}\, |\nabla w|^2\rho,
\end{equation}
for any $w\in H^{1, \, 2}_{0}(\mathcal{D}, \, \rho)$, where we set
$C_S=C_S(t, \, k, \, \rho):=\, C_{t, \, k}  C_\rho^{-t}$.
\end{theorem}

\subsection{Fermi coordinates for $\Omega$}
A key ingredient for the proof of Theorem \ref{th:wcpstrip} is the
choice of suitable local coordinates in order to parametrize the
normal tube $\mathscr{B}(M, \, \e)$; such coordinates are sometimes
called \emph{Fermi coordinates}, cf.  \cite[Chapter 2]{G04} and
\cite[Section 3]{PPS14}.

Up to refinements of the atlas  \eqref{eq:atlas-Gamma}, we can
assume that on each coordinate neighbourhood $U_\alpha \subset
\mathbb{R}^n$ is defined a moving orthonormal frame
$e^1_\alpha,\dots, e^{k}_\alpha$  for the normal bundle
$NM|_{U_{\alpha}}$. This means that, for all $p \in U_{\alpha}$, the
$k$ vectors
\begin{equation*}
  e^1_\alpha(p),\dots, e^{k}_\alpha(p) \in \mathbb{R}^n
 \end{equation*}
are an orthonormal basis for the normal space $N_p M$ and that, for
each $i$, the assignation $p \mapsto e^{i}_\alpha(p)$ defines a
smooth vector field. We take $\mathscr{B}(U_{\alpha}, \,
\varepsilon)$ as in \eqref{eq:B_A}, and we set
\begin{equation*}
\varphi_\alpha=\phi^{-1}_\alpha \colon V_{\alpha} \longrightarrow
U_{\alpha}, \quad E_\alpha^i =e^i_{\alpha} \circ \varphi_\alpha.
\end{equation*}
Thus we have a diffeomorphism
\begin{gather} \label{eq:fermi-coordinates}
  \Phi_\alpha \colon V_{\alpha} \times B_{\mathbb{R}^k}(0,\, \e) \longrightarrow
  \mathscr{B}(U_\alpha, \, \e) \\
  \big( (x^1_\alpha,\dots,x^{n-k}_\alpha),
  (y^1,\dots, y^k)\big) \mapsto \varphi_\alpha
  (x^1_\alpha,\dots,x^{n-k}_\alpha)+
  \sum_{i=1}^{k} y^i  E_\alpha^{i}
  (x^1_\alpha,\dots,x^{n-k}_\alpha),
\end{gather}
where we denoted by  $x^1_\alpha,\dots,x^{n-k}_\alpha$ the Euclidean
coordinates in the open set $V_{\alpha} \subseteq \mathbb{R}^{n-k}$
and by $y^1,\dots, y^k$ the Euclidean coordinates in the open ball
$B_{\mathbb{R}^k}(0, \, \e)$. This shows that
\begin{equation} \label{eq:atlas-fermi}
\left\{ \, \left( \mathscr{B}(U_{\alpha}, \, \varepsilon), \,
\Phi_{\alpha}^{-1} \right) \, \right\}_{\alpha}, \quad
\Phi_\alpha^{-1} \colon
  \mathscr{B}(U_\alpha, \, \e) \stackrel{\simeq}{\longrightarrow}
  V_{\alpha} \times B_{\mathbb{R}^k}(0,\, \e) \subset \mathbb{R}^{n-k}
  \times \mathbb{R}^k
\end{equation}
is an oriented atlas for $\mathscr{B}(M, \, \varepsilon)$.

Denoting by $z^1, \ldots, z^n$ the Euclidean coordinates in
$\mathbb{R}^n$, in the local chart $\mathscr{B}(U_{\alpha}, \, \e)$
we have the Euclidean metric
\begin{equation*}
  g = (dz^1)^2+\cdots + (dz^n)^2.
\end{equation*}
On the other hand,  on $V_\alpha$ we have the induced metric
\begin{equation*}
  h_{\alpha}^\prime = (i_{\alpha} \circ \varphi_{\alpha})^{*} g,
\end{equation*}
where $i_{\alpha} \colon U_{\alpha} \hookrightarrow \mathbb{R}^n$
denotes the embedding map as in Assumption $\mathrm{\ref{eq:h3}}$,
whereas on $B_{\mathbb{R}^k}(0, \, \e)$ we have the metric
$h^{\prime \prime}$ induced by the Euclidean metric on
$\mathbb{R}^k$, namely
\begin{equation*}
  h^{\prime \prime} = (dy^1)^2+\cdots + (dy^{k})^2.
\end{equation*}
The next result allows us to compare the two metrics
$\Phi_\alpha^*g$ and $h_{\alpha}=h_{\alpha}^{\prime}+h^{\prime
\prime}$ on $V_\alpha \times B_{\mathbb{R}^k}(0, \, \e)$, cf.
\cite[Lemma 3.1]{PPS14}.
\begin{lemma}  \label{lemma:metriccomp}
  With the notation introduced above, we have
  \begin{equation} \label{eq:two-metrics}
    \Phi_{\alpha}^*g=h_{\alpha} + \sum_{i=1}^{k}y^i
    (r_{\alpha}^i+2t_{\alpha}^i) + \sum_{i, \,j=1}^{k} y^i y^j
    s_{\alpha}^{ij}\ ,
  \end{equation}
  where $r_{\alpha}^i$, $t_{\alpha}^i$ and $s_{\alpha}^{ij}$ are $2$-tensors
  on $V_{\alpha}
  \times B_{\mathbb{R}^k}(0, \, \e) $ whose coefficients are smooth functions
  on $V_{\alpha}$.
\end{lemma}
\begin{proof}
For simplicity of notation, let us temporarily drop the subscript
$\alpha$. Given any pair $X$, $Y$ of vector fields in $V \times
B_{\mathbb{R}^k}(0, \, \e)$ we have
\begin{equation*}
  (\Phi^*g)(X, \, Y)=g(d\Phi(X),\, d\Phi(Y))=d\Phi(X) \cdot d\Phi(Y),
  \end{equation*}
  where the dot stands for the Euclidean scalar product. Then, in order to
  determine the coefficients
  of $\Phi^*g$ with respect to the frame
  \begin{equation*}
    {\partial}_{x^1},\dots, {\partial}_{ x^{n-k}}, \,
    {\partial}_{y^1},\dots,{\partial}_{y^k}
  \end{equation*}
 of $T V \times TB_{\mathbb{R}^k}(0, \, \e)$, we must compute the quantities
\begin{equation*}
  \begin{split}
    (\Phi^*g)({\partial}_{y^i}, \, {\partial}_{y^j})& =
     d\Phi({\partial}_{y^i}) \cdot  d\Phi({\partial}_{y^j}) =
     {\partial}_{y^i}\Phi \cdot {\partial}_{y^j}\Phi, \\
    (\Phi^*g)({\partial}_{x^a}, \, {\partial}_{y^j})& =
     d\Phi({\partial}_{x^a}) \cdot  d\Phi({\partial}_{y^j}) =
     {\partial}_{x^a}\Phi \cdot {\partial}_{y^j}\Phi, \\
     (\Phi^*g)({\partial}_{x^a}, \, {\partial}_{x^b})& =
     d\Phi({\partial}_{x^a}) \cdot  d\Phi({\partial}_{x^b}) =
     {\partial}_{x^a}\Phi \cdot {\partial}_{x^b}\Phi.
  \end{split}
\end{equation*}
  Observing that $\partial_{x^a} \varphi$ is a tangent vector field
  on $V$, whereas $E^i$ is a normal vector field, we obtain $\partial_{x^a}
  \varphi \cdot E^i=0$.  Moreover, the fact that $e^1, \ldots, e^k$ is an
  orthonormal frame implies $E^i \cdot E^j=\delta_{ij}$.
  Therefore we get
   \begin{equation*}
  \begin{split}
    {\partial}_{y^i}\Phi \cdot {\partial}_{y^j}\Phi & = E^i\cdot E^j =
    \delta_{ij}= h^{\prime \prime}_{ij},\\
     {\partial}_{x^a}\Phi \cdot {\partial}_{y^j}\Phi&  = \sum_{i=1}^{k} y^i
    \partial_{x^a} E^i\cdot E^j
  \end{split}
 \end{equation*}
  and
  \begin{align*}
    {\partial}_{x^a}\Phi \cdot {\partial}_{x^b}\Phi &=
    \partial_{x^a} \varphi \cdot \partial_{x^b}
   \varphi  + \sum_{i=1}^{k} y^i \big(\partial_{x^a} \varphi \cdot
   \partial_{x^b} E^i +
    \partial_{x^b} \varphi \cdot \partial_{x^a} E^i \big)+
    \sum_{i, \, j=1}^{k} y^i y^j \partial_{x^a} E^i
    \cdot
    \partial_{x^b} E^j\\
    & = h^\prime_{ab} + \sum_{i=1}^{k} y^i \big(\partial_{x^a} \varphi
    \cdot \partial_{x^b} E^i +
    \partial_{x^b} \varphi \cdot \partial_{x^a} E^i \big)+
    \sum_{i, \, j=1}^{k} y^i y^j \partial_{x^a} E^i
    \cdot
    \partial_{x^b} E^j.
  \end{align*}
  Now, setting
  \begin{equation} \label{eq:r-t-s}
  \begin{split}
    r^i & = \sum_{a, \, b=1}^{n-k}\big(\partial_{x^a} \varphi \cdot
    \partial_{x^b} E^i +
    \partial_{x^b} \varphi \cdot \partial_{x^a} E^i \big) dx^a dx^b, \\
   t^i & = \sum_{a=1}^{n-k} \sum_{j=1}^{k}\big(
    \partial_{x^a} E^i \cdot E^j \big)dx^a dy^j, \\
    s^{ij} & = \sum_{a, \, b=1}^{n-k} \big( \partial_{x^a} E^i
    \cdot
    \partial_{x^b} E^j \big)dx^a dx^b,  \\
    \end{split}
 \end{equation}
  we obtain \eqref{eq:two-metrics}.
\end{proof}

\begin{example} \label{ex:curves-surfaces}
Let us explicitly compute the expression \eqref{eq:two-metrics} in
three simple cases, namely when $M$ is either a plane curve, or a
space curve or a surface in $\mathbb{R}^3$.
\\ \\
$\boldsymbol{(i)}$ Let $M \subset \mathbb{R}^2$ be a plane curve.
Then $n=2$, $k=1$ and the function $\Phi$ can be chosen of the form
\begin{equation*}
\Phi(x, \, y) = \varphi(x) + y N(x),
\end{equation*}
where $\varphi$ is the arc-length parametrization and $N$ is the
principal normal vector, see \cite[Chapter 1]{dC76}. Then
\begin{equation*}
\partial_x \Phi = \partial_x \varphi + y \partial_x E = (1- \kappa y) T, \quad \partial_y \Phi = N,
\end{equation*}
where $T=T(x)$ is the unit tangent vector and $\kappa=\kappa(x)$ is
the curvature.   From this we get
\begin{equation}  \label{eq:plane-curve}
\begin{split}
\Phi^* g & = (dx)^2 + (dy)^2 - 2y \kappa (dx)^2 + y^2 \kappa^2 (dx)^2 \\
& = h - 2y \kappa (dx)^2 + y^2 \kappa^2 (dx)^2.
\end{split}
\end{equation}
\\
$\boldsymbol{(ii)}$ Let $M \subset \mathbb{R}^3$ be a space curve.
Then $n=3$, $k=2$ and the function $\Phi$ can be chosen of the form
\begin{equation*}
\Phi(x, \, y^1, \, y^2) = \varphi(x)+y^1 E^1(x)+y^2E^2(x),
\end{equation*}
where $\varphi$ is the arc-length parametrization, $E^1$ is the
principal normal vector and $E^2$ is the binormal vector. Thus
$\partial_x \varphi= T$, where $T=T(x)$ is the unit tangent vector,
and by the Fr\'enet formulas \cite[p. 19]{dC76} we obtain
\begin{equation*}
\partial_x E^1 = -\kappa T-\tau E^2, \quad \partial_x E^2 = \tau E^1,
\end{equation*}
 where $\kappa=\kappa(x)$ is the curvature and $\tau=\tau(x)$ is the torsion. Now, using \eqref{eq:r-t-s}, we easily find
\begin{align*}
r^1 & = -2 \kappa (dx)^2, & t^1& = - \tau dx dy^2,  & r^2 &= 0, &
t^2 & = \tau
dx dy^1, \\
s^{11}& = (\kappa^2 + \tau^2) (dx)^2, & s^{12}&=0, & s^{22} & =
\tau^2 (dx)^2,
\end{align*}
so we get
\begin{equation} \label{eq:space-curve}
\Phi^*g = h -2 y^1 ( \kappa (dx)^2 + \tau dx dy^2) + 2 y^2 ( \tau dx
dy^1) + (y^1)^2 ((\kappa^2 + \tau^2) (dx)^2) + (y^2)^2 (\tau^2
(dx)^2).
\end{equation}
\\
$\boldsymbol{(iii)}$ Let $M \subset \mathbb{R}^3$ be a surface. Then
$n=3$, $k=1$ and the function $\Phi$ can be chosen of the form
\begin{equation*}
\Phi(x^1, \, x^2, \, y) = \varphi(x^1, \, x^2) + yE(x^1, \, x^2),
\end{equation*}
where we take $E(x^1, \, x^2) = \frac{ \partial_{x^1} \varphi \wedge
\partial_{x^2} \varphi}{ \|  \partial_{x^1} \varphi \wedge \partial_{x^2}
\varphi \|}$. Denoting as customary by
\begin{equation*}
\mathsf{E}=\partial _{x^1} \varphi \cdot \partial _{x^1} \varphi,
\quad \mathsf{F}= \partial _{x^1} \varphi \cdot \partial _{x^2}
\varphi, \quad \mathsf{G}=\partial _{x^2} \varphi \cdot \partial
_{x^2} \varphi
\end{equation*}
the coefficients of the first fundamental form $\mathrm{I}_M$ of $M$
and by
\begin{equation*}
\mathsf{e}=-\partial_{x^1} E \cdot \partial _{x^1} \varphi, \quad
\mathsf{f}= - \partial _{x^2} E \cdot \partial _{x^1} \varphi, \quad
\mathsf{g}=-\partial _{x^2} E \cdot \partial _{x^2} \varphi
\end{equation*}
the coefficients of the second fundamental form $\mathrm{II}_M$,
from the formulas in \cite[p. 92 and p. 154]{dC76} we deduce
\begin{align*} \label{eq:space-surfaces}
 \partial_{x^1} \Phi \cdot \partial_{x^1} \Phi & = \mathsf{E} -2 y \mathsf{e}
 + y^2\psi_{11}, &
 \partial_{x^1} \Phi \cdot \partial_{x^2} \Phi & = \mathsf{F} -2 y \mathsf{f}
 + y^2\psi_{12}, &
 \partial_{x^2} \Phi \cdot \partial_{x^2} \Phi & = \mathsf{G} -2 y \mathsf{g}
 + y^2\psi_{22}, \\
 \partial_{x^1} \Phi \cdot \partial_{y} \Phi& =0 , &
 \partial_{x^2} \Phi \cdot \partial_{y} \Phi& =0 , &
 \partial_{y} \Phi \cdot \partial_{y} \Phi& =1,
\end{align*}
where
\begin{equation*}
\psi_{11} = \frac{\mathsf{e}^2 \mathsf{G} - 2 \mathsf{e} \mathsf{f}
\mathsf{F}+ \mathsf{f}^2 \mathsf{E}}{\mathsf{E} \mathsf{G}-
\mathsf{F}^2}, \quad \psi_{12} = \frac{-\mathsf{f}^2 \mathsf{F}+
\mathsf{e} \mathsf{f} \mathsf{G}- \mathsf{e} \mathsf{g} \mathsf{F}+
\mathsf{g} \mathsf{f} \mathsf{E}}{\mathsf{E} \mathsf{G}-
\mathsf{F}^2}, \quad \psi_{22} = \frac{\mathsf{f}^2 \mathsf{G} - 2
\mathsf{f} \mathsf{g} \mathsf{F}+ \mathsf{g}^2
\mathsf{E}}{\mathsf{E} \mathsf{G}- \mathsf{F}^2}.
\end{equation*}
Then, using the expressions of the mean curvature $H_M$ and the
Gaussian curvature $K_M$ of $M$ in terms of the coefficients of the
first two fundamental forms (cf. \cite[p. 155-156]{dC76}), we can
rewrite \eqref{eq:two-metrics} as
\begin{equation} \label{eq:two-metric-surfaces}
\Phi^*g = \mathrm{I}_M + (dy)^2 -2 y \mathrm{II}_M + y^2
\mathrm{III}_M,
\end{equation}
where $\mathrm{I}_M+(dy)^2=h$ and $\mathrm{III}_M = -K_M
\mathrm{I}_M+ 2H_M \mathrm{II}_M$ is the so-called \emph{third
fundamental form} of $M$, see \cite[p. 98]{St89}.
\end{example}

\medskip

As a consequence of Lemma \ref{lemma:metriccomp} we have the result
below, comparing the volume forms of the two metrics
$\Phi_\alpha^*g$ and $h_\alpha$ on $V_\alpha\times B_{
\mathbb{R}^k}(0, \, \e)$.

\begin{corollary} \label{cor:dvol}
  With the notation introduced above, we have
  \begin{equation} \label{eq:two-dvol}
    \mathrm{dvol}_{\Phi_\alpha^*g} = \lambda_\alpha\ \mathrm{dvol}_{h_\alpha}
  \end{equation}
  for a smooth function $\lambda$ on $\mathscr{B}(M, \, \varepsilon)$ such that
  $\lambda \equiv 1$ on $M$. If moreover $\varepsilon <1$, for every
  \begin{equation*}
    (x_\alpha, \, y)= \left( (x_\alpha^1,\ldots, x_\alpha^{n-k}), \,
    (y^1,\dots,y^k) \right)\in
    V_{\alpha }\times B_{\mathbb{R}^k}(0, \, \varepsilon)
  \end{equation*}
  we have
  \begin{equation}
    \vert \lambda_\alpha(x_\alpha,y) -1 \vert \le K_1 \norm{y},
    \label{eq:estimate}
  \end{equation}
  where $K_1$ is a positive constant, independent on $\alpha$.

 It follows that there exists $\varepsilon_1=\min \bigg\{1, \,
   \dfrac{1}{2K_1} \bigg\}$,
  independent on $\alpha$, such that if
  $\varepsilon < \varepsilon_1$ then
\begin{equation} \label{eq:varepsilon_0}
\frac{1}{2} \leq \lambda_\alpha(x_\alpha, \, y) \leq \frac{3}{2}
\quad \textrm{for all} \;\; (x_\alpha, \, y) \in V_{\alpha} \times
B_{\mathbb{R}^k}(0, \, \varepsilon).
\end{equation}
 \end{corollary}
\begin{proof}
  In order to simplify the notation,
 we drop the subscript $\alpha$. Since $\det h \neq 0$
everywhere, by Lemma \ref{lemma:metriccomp} we can
  write the following finite expansion of $\det \Phi^* g$ in terms of the
  normal coordinates $y^i$:
  \begin{equation} \label{eq:expansion}
    \det \Phi^* g = \det h + \sum_{i=1}^k y^i A^i + \sum_{i, \, j=1}^k
    y^{i}y^{j}A^{ij} +  \sum_{i, \, j, \, l \ =1}^k y^{i}y^{j}y^{l}A^{ijl}
    + \ldots
  \end{equation}
  where all the coefficients are smooth functions depending on the first and second
  derivatives of $\Phi$. Now $\varepsilon <1$  implies $\norm{y} <1$ and
  so $|y^i| <1$ for all $i$, thus from \eqref{eq:expansion}
  we obtain
  \begin{equation} \label{eq:det-2}
    \det \Phi^* g  =  \det h +  \sum_{i=1}^k \widetilde{K}^{i}_1(x,y) y^i
    % \leq |\det h|
    % + k \widetilde{K}_1 \norm{y},
  \end{equation}
  where $\vert \widetilde{K}^{i}_1(x,y) \vert \le \widetilde{K}_1$ and
  $\widetilde{K}_1$ is a positive constant, which is moreover independent on the chart
  because of Assumption  $\mathrm{\ref{eq:h3}}$. Setting $\lambda(x, \, y) :=\det \Phi^* g \cdot
      \det h ^{-1}$, from \eqref{eq:det-2} we get \footnote{Here we use the
    inequality $ \vert \sqrt{1+x} -1 \vert \le \vert x \vert$ for
  $x\ge -1$.}
  \begin{equation*}
    \begin{split}
    \vert \lambda(x, \, y) -1 \vert & = \vert \sqrt{ \det \Phi^* g \cdot
      \det h ^{-1}} -1 \vert = \left\vert \sqrt{1 + \sum_{i=1}^k
      \dfrac{\widetilde{K}^{i}_1(x,y)}{\det h} y^i} -1 \right\vert  \\
      & \leq \left\vert \sum_{i=1}^k
      \dfrac{\widetilde{K}^{i}_1(x,y)}{\det h} y^i \right\vert \le
      \sum_{i=1}^k  \left\vert \dfrac{\widetilde{K}^{i}_1(x,y)}{\det h}
      \right\vert \vert y^i \vert \leq \sum_{i=1}^k
      \dfrac{\widetilde{K}_1}{\det h} \vert y^i \vert \le k
      \dfrac{\widetilde{K}_1}{\det h} \norm{y} = K_1 \norm{y}\ ,
    \end{split}
  \end{equation*}
  where we put $K_1:= k\widetilde{K}_1/\det h$. Again by Assumption $\mathrm{\ref{eq:h3}}$, the function $\det h$ is bounded, uniformly with respect to $\alpha$, so the
  positive constant
  $K_1$ is also independent on the chart. The proof of the last statement
  is now straightforward.
\end{proof}

\begin{remark} \label{rmk:dvol-h}
Using the same arguments as in the proof of Corollary
\ref{cor:dvol}, we can show that on $V_{\alpha}$ we have
\begin{equation} \label{eq:dvol-h}
  \mathrm{dvol}_{h_\alpha^{'}} = \mu_\alpha(x_\alpha)\, dx_\alpha
\end{equation}
for a smooth positive function $\mu_\alpha(x_\alpha)$, uniformly
bounded with respect to $\alpha$.
\end{remark}

\section{The weak comparison principle: Proof of Theorem
    \ref{th:wcpstrip}} \label{section:proof}

    \subsection{Test functions} \label{subsection:testfunctions}
  For the sake of brevity, we write $B_R$ instead of $B_{M}(\bar{p}, \,
  R)$, see Assumption $\mathrm{\ref{eq:h4}}$.
    Let $\varphi_R\in C^{\infty}_c
    (M) $ be such that
    \begin{equation}\label{Eq:Cut-off1}
      \begin{cases}
    0 \leq \varphi_R \leq 1 & \text{ in } M, \\
    \varphi_R \equiv 1 & \text{ in } B_R,\\
    \varphi_R \equiv 0 & \text{ in } M \setminus
    B_{2R},\\
    |\nabla \varphi_R | \leq \frac 2R & \text{ in }  B_{2R} \setminus B_{R}.
      \end{cases}
    \end{equation}
    For instance we could take $\varphi_R$ of the form
    \begin{equation*}
    \varphi_R(p) =
    \begin{cases}
       1 & \text{ in } B_{R}, \\
    \left( \frac{2R-d_M(\bar p , \, p)}{R}\right)^2    & \text{ in } B_{2R}
    \setminus B_{R}, \\
       0 & \text{ in } M\setminus B_{2R},
     \end{cases}
    \end{equation*}
    in fact the distance function $d_M(\bar p , \, p)$ has
    distributional gradient whose $L^\infty$-norm is at most $1$, see
    \cite[Theorem 11.3 p. 296]{Gr09}.
    Since $\mathscr{B}(M, \,
    \varepsilon)$ is locally trivial, in view of diffeomorphisms
    in \eqref{eq:fermi-coordinates},
    we can extend $\varphi_R$ to the whole tubular
    neighbourhood and
    we continue to denote this extension by
    $\varphi_R$.
    In what follows we will consider the bounded domain
\begin{equation*}
\Omega_R:= \Omega\cap  \mathscr{B}(B_{R}, \, \e),
\end{equation*}
see Figure \ref{fig:omega}  below.
\begin{center}
\begin{tikzpicture}[line cap=round,line join=round,>=triangle 45,x=0.6 cm,y=0.5cm]
\clip(-4.1066640247289445,-6.069139397543109) rectangle
(16.002791730848084,6.949871217548806); \fill[line
width=1.2pt,color=qqwwtt,fill=qqwwtt,pattern=north east
lines,pattern color=qqwwtt] (-1.9659402406656563,3) --
(-1.9659402406656563,-3) -- (4.971184070787558,-3) -- (5,3) --
cycle;\fill[line width=0.4pt,color=ttzzqq,fill=ttzzqq,fill
opacity=0.47] (4.971184070787558,-3) -- (-1.9659402406656563,-3) --
(-3,-3) -- (-12.716254027339456,-3) -- (-12.648591438653476,3) --
(-3,3) -- (-1.9659402406656563,3) -- (1.5419165513808526,3) -- (5,3)
-- (5.25880771850801,2.9853017354668414) --
(5.551692703766248,2.9793244908697343) --
(5.79675973224763,2.967370001675521) --
(6.021742944519114,2.9453672328575173) --
(6.264387610015966,2.912279323926129) --
(6.496002972535688,2.8791914149947404) --
(6.716589032078281,2.8461035060633515) --
(6.959233697575133,2.8240449001090924) --
(7.179819757117726,2.8019862941548337) --
(7.3562886047518,2.790956991177704) --
(7.521498410640398,2.7894793034349874) --
(7.709226300019949,2.7799276882005746) --
(7.874665844676893,2.768898385223445) --
(8.062163995288097,2.7578690822463154) --
(8.238632842922172,2.7358104762920563) --
(8.448189599487634,2.713751870337797) --
(8.635687750098839,2.691693264383538) --
(8.823185900710042,2.6806639614064087) --
(9.010684051321247,2.6696346584292794) --
(9.20921150490958,2.6586053554521496) --
(9.385680352543654,2.6475760524750203) --
(9.562149200177728,2.6365467494978905) --
(9.749647350788932,2.614488143543632) --
(9.926116198423006,2.5924295375893727) --
(10.11361434903421,2.5703709316351135) --
(10.290083196668284,2.537283022703725) --
(10.455522741325229,2.482136507818077) --
(10.609932983005045,2.4269899929324295) --
(10.75331392170773,2.3828727810239116) --
(10.885665557433285,2.3166969631611343) --
(11.006987890181712,2.2505211452983573) --
(11.150368828884396,2.162286721481321) --
(11.293749767587082,2.074052297664285) --
(11.426101403312638,1.996847176824378) --
(11.558453039038193,1.8865541470530829) --
(11.690804674763749,1.7983197232360466) --
(11.801097704535046,1.72111460239614) --
(11.91798412170466,1.6157112478213589) --
(12.043742370031898,1.5005285428535495) --
(12.151925795794947,1.3997650871226337) --
(12.231240520643102,1.279942483310959) --
(12.29741633850588,1.1255322416311455) --
(12.356297724469144,0.9639929873863685) --
(12.385867469885234,0.7519266050264581) --
(12.407709368277176,0.5961256987289283) --
(12.403862983276795,0.44600287736993083) --
(12.403862983276795,0.23005671667120567) -- (12.386402515436012,0)
-- (12.352562853391527,-0.19798411562439758) --
(12.28638703552875,-0.42959947814411764) --
(12.209181914688843,-0.628126931732449) --
(12.115934769011826,-0.8136830600392994) --
(11.966556679685327,-1.0335986504353005) --
(11.861207405760402,-1.1865405414645167) --
(11.668561199182376,-1.3861791237433008) --
(11.454710297252582,-1.565684025640039) --
(11.180168072650675,-1.7494497564004416) --
(10.911863936927713,-1.8530579117650317) --
(10.630137027512346,-1.9108451871223033) --
(10.366737813184793,-2.003751930445969) --
(10.190414836114842,-2.10936002423165) --
(10.053127728032026,-2.2686237785800416) --
(9.920482135241434,-2.4512747786712987) --
(9.786734589257037,-2.5898700183802172) --
(9.659584142640462,-2.7254971614378976) --
(9.506636131403747,-2.82703731957157) --
(9.328992981437361,-2.9204611795833135) --
(9.134028963291943,-2.996751447553259) --
(8.87972807005879,-3.047611626199889) --
(8.642800684440632,-3.0451267205927928) --
(8.3965563729158,-3.039134929758784) --
(8.142255479682646,-3.039134929758784) --
(7.845571104243967,-3.039134929758784) --
(7.531933335923077,-3.039134929758784) --
(7.218295567602188,-3.0306582333176793) --
(6.8769689120383815,-3.0175308822286664) --
(6.459162246016318,-3.002190553042354) --
(6.073941548052997,-3.0052281439943638) --
(5.743350386849897,-2.996751447553259) --
(5.3873291363234825,-2.996751447553259) -- cycle;\draw [line
width=2pt,color=ffqqqq,domain=-4.1066640247289445:16.002791730848084]
plot(\x,{(-0-0*\x)/1});\draw [line width=1.2pt,
color=wrwrwr,domain=-4.1066640247289445:16.002791730848084]
plot(\x,{(--5-0*\x)/1});\draw [line
width=1.2pt,color=wrwrwr,domain=-4.1066640247289445:16.002791730848084]
plot(\x,{(-5-0*\x)/1});\draw [line
width=1.2pt,color=ttzzqq,domain=-4.1066640247289445:-2]
plot(\x,{(-3-0*\x)/-1});\draw [line
width=1.2pt,color=ttzzqq,domain=-4.1066640247289445:-2]
plot(\x,{(--3-0*\x)/-1});\draw [line width=1pt, dashed] (14, 5)--
(14,0);\draw (10.096117407356337,-3.803089715137713)
node[anchor=north west] {$\mathcal{B}(M, \, \varepsilon)$};\draw
[line width=1.2pt,color=qqwwtt] (-1.9659402406656563,3)--
(-1.9659402406656563,-3);\draw [line width=1.2pt,color=qqwwtt]
(-1.9659402406656563,-3)-- (4.971184070787558,-3);\draw [line
width=1.2pt,color=qqwwtt] (4.971184070787558,-3)-- (5,3);\draw [line
width=1.2pt,color=qqwwtt] (5,3)-- (-1.9659402406656563,3);\draw
[line width=0.4pt,color=qqwwtt] (4.971184070787558,-3)--
(-1.9659402406656563,-3);\draw [line width=0.4pt,color=ttzzqq]
(-3,-3)-- (-12.716254027339456,-3);\draw [line
width=0.4pt,color=ttzzqq] (-12.716254027339456,-3)--
(-12.648591438653476,3);\draw [line width=0.4pt,color=ttzzqq]
(-12.648591438653476,3)-- (-3,3);\draw [line
width=0.4pt,color=ttzzqq] (-3,3)-- (-1.9659402406656563,3);\draw
[line width=0.4pt,color=ttzzqq] (-1.9659402406656563,3)--
(1.5419165513808526,3);\draw [line width=0.4pt,color=ttzzqq]
(1.5419165513808526,3)-- (5,3);\draw [line width=0.4pt,color=ttzzqq]
(5,3)-- (5.25880771850801,2.9853017354668414);\draw [line
width=0.4pt,color=ttzzqq] (5.25880771850801,2.9853017354668414)--
(5.551692703766248,2.9793244908697343);\draw [line
width=0.4pt,color=ttzzqq] (5.551692703766248,2.9793244908697343)--
(5.79675973224763,2.967370001675521);\draw [line
width=0.4pt,color=ttzzqq] (5.79675973224763,2.967370001675521)--
(6.021742944519114,2.9453672328575173);\draw [line
width=0.4pt,color=ttzzqq] (6.021742944519114,2.9453672328575173)--
(6.264387610015966,2.912279323926129);\draw [line
width=0.4pt,color=ttzzqq] (6.264387610015966,2.912279323926129)--
(6.496002972535688,2.8791914149947404);\draw [line
width=0.4pt,color=ttzzqq] (6.496002972535688,2.8791914149947404)--
(6.716589032078281,2.8461035060633515);\draw [line
width=0.4pt,color=ttzzqq] (6.716589032078281,2.8461035060633515)--
(6.959233697575133,2.8240449001090924);\draw [line
width=0.4pt,color=ttzzqq] (6.959233697575133,2.8240449001090924)--
(7.179819757117726,2.8019862941548337);\draw [line
width=0.4pt,color=ttzzqq] (7.179819757117726,2.8019862941548337)--
(7.3562886047518,2.790956991177704);\draw [line
width=0.4pt,color=ttzzqq] (7.3562886047518,2.790956991177704)--
(7.521498410640398,2.7894793034349874);\draw [line
width=0.4pt,color=ttzzqq] (7.521498410640398,2.7894793034349874)--
(7.709226300019949,2.7799276882005746);\draw [line
width=0.4pt,color=ttzzqq] (7.709226300019949,2.7799276882005746)--
(7.874665844676893,2.768898385223445);\draw [line
width=0.4pt,color=ttzzqq] (7.874665844676893,2.768898385223445)--
(8.062163995288097,2.7578690822463154);\draw [line
width=0.4pt,color=ttzzqq] (8.062163995288097,2.7578690822463154)--
(8.238632842922172,2.7358104762920563);\draw [line
width=0.4pt,color=ttzzqq] (8.238632842922172,2.7358104762920563)--
(8.448189599487634,2.713751870337797);\draw [line
width=0.4pt,color=ttzzqq] (8.448189599487634,2.713751870337797)--
(8.635687750098839,2.691693264383538);\draw [line
width=0.4pt,color=ttzzqq] (8.635687750098839,2.691693264383538)--
(8.823185900710042,2.6806639614064087);\draw [line
width=0.4pt,color=ttzzqq] (8.823185900710042,2.6806639614064087)--
(9.010684051321247,2.6696346584292794);\draw [line
width=0.4pt,color=ttzzqq] (9.010684051321247,2.6696346584292794)--
(9.20921150490958,2.6586053554521496);\draw [line
width=0.4pt,color=ttzzqq] (9.20921150490958,2.6586053554521496)--
(9.385680352543654,2.6475760524750203);\draw [line
width=0.4pt,color=ttzzqq] (9.385680352543654,2.6475760524750203)--
(9.562149200177728,2.6365467494978905);\draw [line
width=0.4pt,color=ttzzqq] (9.562149200177728,2.6365467494978905)--
(9.749647350788932,2.614488143543632);\draw [line
width=0.4pt,color=ttzzqq] (9.749647350788932,2.614488143543632)--
(9.926116198423006,2.5924295375893727);\draw [line
width=0.4pt,color=ttzzqq] (9.926116198423006,2.5924295375893727)--
(10.11361434903421,2.5703709316351135);\draw [line
width=0.4pt,color=ttzzqq] (10.11361434903421,2.5703709316351135)--
(10.290083196668284,2.537283022703725);\draw [line
width=0.4pt,color=ttzzqq] (10.290083196668284,2.537283022703725)--
(10.455522741325229,2.482136507818077);\draw [line
width=0.4pt,color=ttzzqq] (10.455522741325229,2.482136507818077)--
(10.609932983005045,2.4269899929324295);\draw [line
width=0.4pt,color=ttzzqq] (10.609932983005045,2.4269899929324295)--
(10.75331392170773,2.3828727810239116);\draw [line
width=0.4pt,color=ttzzqq] (10.75331392170773,2.3828727810239116)--
(10.885665557433285,2.3166969631611343);\draw [line
width=0.4pt,color=ttzzqq] (10.885665557433285,2.3166969631611343)--
(11.006987890181712,2.2505211452983573);\draw [line
width=0.4pt,color=ttzzqq] (11.006987890181712,2.2505211452983573)--
(11.150368828884396,2.162286721481321);\draw [line
width=0.4pt,color=ttzzqq] (11.150368828884396,2.162286721481321)--
(11.293749767587082,2.074052297664285);\draw [line
width=0.4pt,color=ttzzqq] (11.293749767587082,2.074052297664285)--
(11.426101403312638,1.996847176824378);\draw [line
width=0.4pt,color=ttzzqq] (11.426101403312638,1.996847176824378)--
(11.558453039038193,1.8865541470530829);\draw [line
width=0.4pt,color=ttzzqq] (11.558453039038193,1.8865541470530829)--
(11.690804674763749,1.7983197232360466);\draw [line
width=0.4pt,color=ttzzqq] (11.690804674763749,1.7983197232360466)--
(11.801097704535046,1.72111460239614);\draw [line
width=0.4pt,color=ttzzqq] (11.801097704535046,1.72111460239614)--
(11.91798412170466,1.6157112478213589);\draw [line
width=0.4pt,color=ttzzqq] (11.91798412170466,1.6157112478213589)--
(12.043742370031898,1.5005285428535495);\draw [line
width=0.4pt,color=ttzzqq] (12.043742370031898,1.5005285428535495)--
(12.151925795794947,1.3997650871226337);\draw [line
width=0.4pt,color=ttzzqq] (12.151925795794947,1.3997650871226337)--
(12.231240520643102,1.279942483310959);\draw [line
width=0.4pt,color=ttzzqq] (12.231240520643102,1.279942483310959)--
(12.29741633850588,1.1255322416311455);\draw [line
width=0.4pt,color=ttzzqq] (12.29741633850588,1.1255322416311455)--
(12.356297724469144,0.9639929873863685);\draw [line
width=0.4pt,color=ttzzqq] (12.356297724469144,0.9639929873863685)--
(12.385867469885234,0.7519266050264581);\draw [line
width=0.4pt,color=ttzzqq] (12.385867469885234,0.7519266050264581)--
(12.407709368277176,0.5961256987289283);\draw [line
width=0.4pt,color=ttzzqq] (12.407709368277176,0.5961256987289283)--
(12.403862983276795,0.44600287736993083);\draw [line
width=0.4pt,color=ttzzqq] (12.403862983276795,0.44600287736993083)--
(12.403862983276795,0.23005671667120567);\draw [line
width=0.4pt,color=ttzzqq] (12.403862983276795,0.23005671667120567)--
(12.386402515436012,0);\draw [line width=0.4pt,color=ttzzqq]
(12.386402515436012,0)--
(12.352562853391527,-0.19798411562439758);\draw [line
width=0.4pt,color=ttzzqq]
(12.352562853391527,-0.19798411562439758)--
(12.28638703552875,-0.42959947814411764);\draw [line
width=0.4pt,color=ttzzqq] (12.28638703552875,-0.42959947814411764)--
(12.209181914688843,-0.628126931732449);\draw [line
width=0.4pt,color=ttzzqq] (12.209181914688843,-0.628126931732449)--
(12.115934769011826,-0.8136830600392994);\draw [line
width=0.4pt,color=ttzzqq] (12.115934769011826,-0.8136830600392994)--
(11.966556679685327,-1.0335986504353005);\draw [line
width=0.4pt,color=ttzzqq] (11.966556679685327,-1.0335986504353005)--
(11.861207405760402,-1.1865405414645167);\draw [line
width=0.4pt,color=ttzzqq] (11.861207405760402,-1.1865405414645167)--
(11.668561199182376,-1.3861791237433008);\draw [line
width=0.4pt,color=ttzzqq] (11.668561199182376,-1.3861791237433008)--
(11.454710297252582,-1.565684025640039);\draw [line
width=0.4pt,color=ttzzqq] (11.454710297252582,-1.565684025640039)--
(11.180168072650675,-1.7494497564004416);\draw [line
width=0.4pt,color=ttzzqq] (11.180168072650675,-1.7494497564004416)--
(10.911863936927713,-1.8530579117650317);\draw [line
width=0.4pt,color=ttzzqq] (10.911863936927713,-1.8530579117650317)--
(10.630137027512346,-1.9108451871223033);\draw [line
width=0.4pt,color=ttzzqq] (10.630137027512346,-1.9108451871223033)--
(10.366737813184793,-2.003751930445969);\draw [line
width=0.4pt,color=ttzzqq] (10.366737813184793,-2.003751930445969)--
(10.190414836114842,-2.10936002423165);\draw [line
width=0.4pt,color=ttzzqq] (10.190414836114842,-2.10936002423165)--
(10.053127728032026,-2.2686237785800416);\draw [line
width=0.4pt,color=ttzzqq] (10.053127728032026,-2.2686237785800416)--
(9.920482135241434,-2.4512747786712987);\draw [line
width=0.4pt,color=ttzzqq] (9.920482135241434,-2.4512747786712987)--
(9.786734589257037,-2.5898700183802172);\draw [line
width=0.4pt,color=ttzzqq] (9.786734589257037,-2.5898700183802172)--
(9.659584142640462,-2.7254971614378976);\draw [line
width=0.4pt,color=ttzzqq] (9.659584142640462,-2.7254971614378976)--
(9.506636131403747,-2.82703731957157);\draw [line
width=0.4pt,color=ttzzqq] (9.506636131403747,-2.82703731957157)--
(9.328992981437361,-2.9204611795833135);\draw [line
width=0.4pt,color=ttzzqq] (9.328992981437361,-2.9204611795833135)--
(9.134028963291943,-2.996751447553259);\draw [line
width=0.4pt,color=ttzzqq] (9.134028963291943,-2.996751447553259)--
(8.87972807005879,-3.047611626199889);\draw [line
width=0.4pt,color=ttzzqq] (8.87972807005879,-3.047611626199889)--
(8.642800684440632,-3.0451267205927928);\draw [line
width=0.4pt,color=ttzzqq] (8.642800684440632,-3.0451267205927928)--
(8.3965563729158,-3.039134929758784);\draw [line
width=0.4pt,color=ttzzqq] (8.3965563729158,-3.039134929758784)--
(8.142255479682646,-3.039134929758784);\draw [line
width=0.4pt,color=ttzzqq] (8.142255479682646,-3.039134929758784)--
(7.845571104243967,-3.039134929758784);\draw [line
width=0.4pt,color=ttzzqq] (7.845571104243967,-3.039134929758784)--
(7.531933335923077,-3.039134929758784);\draw [line
width=0.4pt,color=ttzzqq] (7.531933335923077,-3.039134929758784)--
(7.218295567602188,-3.0306582333176793);\draw [line
width=0.4pt,color=ttzzqq] (7.218295567602188,-3.0306582333176793)--
(6.8769689120383815,-3.0175308822286664);\draw [line
width=0.4pt,color=ttzzqq] (6.8769689120383815,-3.0175308822286664)--
(6.459162246016318,-3.002190553042354);\draw [line
width=0.4pt,color=ttzzqq] (6.459162246016318,-3.002190553042354)--
(6.073941548052997,-3.0052281439943638);\draw [line
width=0.4pt,color=ttzzqq] (6.073941548052997,-3.0052281439943638)--
(5.743350386849897,-2.996751447553259);\draw [line
width=0.4pt,color=ttzzqq] (5.743350386849897,-2.996751447553259)--
(5.3873291363234825,-2.996751447553259);\draw [line
width=0.4pt,color=ttzzqq] (5.3873291363234825,-2.996751447553259)--
(4.971184070787558,-3);\draw[color=ffqqqq]
(7.0810365123332145,0.5128702169669421) node {${M}$};\draw
[fill=black] (1.5624374965153105,0) circle
(2.5pt);\draw[color=black] (1.8219014498782096,-0.5742205339156638)
node {$\overline{p}$};\draw[color=DarkGreen]
(-0.054981087308979895,-1.4471412682706412) node
{$\boldsymbol{\Omega_R}$}; \draw[color=black] (14.3,2.7)
node{$\varepsilon$}; \draw[color=NavyBlue] (9.163601615154606,-
1.1087059430581054) node {$\boldsymbol{\Omega}$}; \draw[line
width=2pt, color=RoyalPurple] (-1.9359402406656563,0)-- (4.9,0);
\draw[color=RoyalPurple] (3.36, 0.5) node {$\boldsymbol{B_R}$};
\end{tikzpicture}

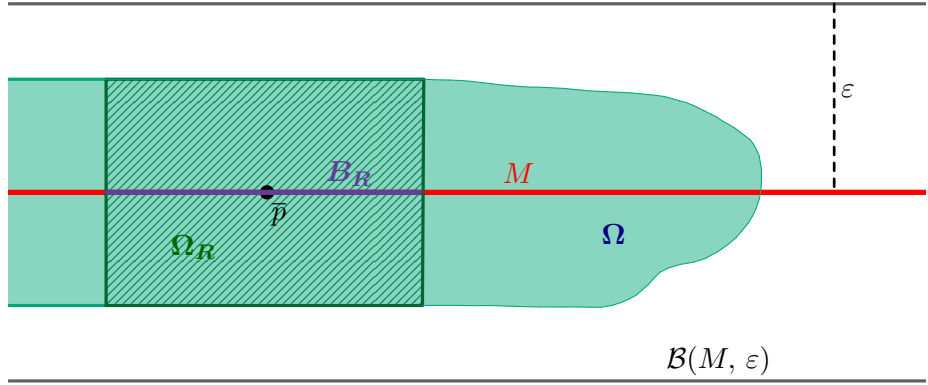
\captionof{figure}{The bounded domain $\Omega_R$.} \label{fig:omega}
\end{center}
\bigskip

Moreover, for a fixed real number $\beta \geq 1$, we will define
    \begin{equation}\label{Eq:Cut-off}
      \psi_R:=[(u-v)^+]^{\beta}\varphi_R^2.
    \end{equation}
With this notation, we can state the following crucial result, whose
proof relies on our assumption that $u \leq v$ on $\partial \Omega$.
\begin{lemma} \label{lem:plugging-psi}
The function $\psi_R$ belongs to $H^1_0\left(\Omega_{2R}\right)$,
and it can be plugged into \eqref{weak1} and \eqref{weak2}.
\end{lemma}
\begin{proof}
We have $\partial \Omega_{2R} \subseteq \partial \Omega  \cup
\partial \mathscr{B}(B_{2R}, \, \e)$; furthermore, by assumption, $(u-v)^+=0$
on $\partial \Omega$, whereas $\varphi_R=0$ on $\partial
\mathscr{B}(B_{2R}, \, \e)$. This means that $\psi_R=0$ on $\partial
\Omega_{2R}$, so the first claim follows from \cite[Chapter
7]{GT83}, cf. also \cite[Theorem 8.12]{Bre}. The second claim is an
immediate consequence of the density of
$H^1_0\left(\Omega_{2R}\right)$ in $C^{\infty}_c(\Omega_{2R})$.
\end{proof}

Plugging $\psi_R$ as a test function into \eqref{weak1} and
\eqref{weak2}
    and subtracting we get
\begin{equation} \label{eq:ABCL}
    \begin{split}
      &\beta\,\,
      \underset{\mathcal{L}_{2R}}{
    \underbrace{\int_{\Omega_{2R}}\,a(z) \,
    [(u-v)^+]^{\beta-1}|\nabla(u-v)|^2 \, \varphi_R^2\, \mathrm{dvol}_g}}\\
    &\leq \int_{\Omega_{2R}}|\Lambda|\; [(u-v)^+]^{\beta} \; \left||\nabla u|^q-|\nabla
    v|^q\right|\, \varphi_R^2\, \mathrm{dvol}_g
    \,+\,\int_{\Omega_{2R}}|f(z,u)-f(z,v)| \;[(u-v)^+]^{\beta}\varphi_R^2\, \mathrm{dvol}_g\\
    &+2\int_{\Omega_{2R}}a(z) \, [(u-v)^+]^{\beta}\, |\nabla(u-v)| \,|\nabla\varphi_R|\,\varphi_R \, \mathrm{dvol}_g\\
    &\leq|\Lambda| \, q \; (\|\nabla u\|_\infty+\|\nabla v\|_\infty)^{q-1}
    \underset{\mathcal{A}_{2R}}{\underbrace{\int_{\Omega_{2R}} [(u-v)^+]^{\beta} \, |\nabla
    (u-v)| \, \varphi_R^2\, \mathrm{dvol}_g}}
    \\
    &+\,L_f \,
    \underset{\mathcal{B}_{2R}}{\underbrace{\int_{\Omega_{2R}}[(u-v)^+]^{\beta+1}
      \varphi_R^2\, \mathrm{dvol}_g}}\\
      &
      + 2\,\underset{\mathcal{C}_{2R}}{\underbrace{\int_{\Omega_{2R}}a(z)\, [(u-v)^+]^{\beta} \, |\nabla (u-v)| \,|\nabla\varphi_R|\, \varphi_R \, \mathrm{dvol}_g}},
    \end{split}
  \end{equation}
  where we estimated  $\left||\nabla u|^q-|\nabla v|^q\right|$ via the Lagrange
  theorem applied to the function $t^q$ and we
  exploited the reverse triangular inequality. Moreover the constant
  $L_f = L_f(\| u\|_\infty+\| v\|_\infty)$ is the one contained in Assumption $\mathrm{\ref{eq:h2}}$. \\
  In the sequel, we will estimate the terms $\mathcal{A}_{2R}$, $\mathcal{B}_{2R}$ and
  $\mathcal{C}_{2R}$ in \eqref{eq:ABCL}, making repeated use of the Young inequality
  \begin{equation*}
    ab \leq  \tau a^2 + \frac{1}{4\tau} b^2, \quad a, \, b\in\mathbb{R},\quad
    \tau>0.
  \end{equation*}

\subsection{Estimate for $\mathcal{A}_{2R}$}

Dealing first with the term $\mathcal{A}_{2R}$, we obtain
\begin{equation} \label{eq:A2r-1}
    \begin{split}
      \mathcal{A}_{2R} & =
      \int_{\Omega_{2R}} [(u-v)^+]^{\beta} \, |\nabla (u-v)| \, \varphi_R^2\, \mathrm{dvol}_g\\
      &=\int_{\Omega_{2R}}\,a(z)^{\frac{1}{2}}\, [(u-v)^+]^{\frac{\beta-1}{2}}\, |\nabla
      (u-v)| \, \varphi_R\,a(z)^{-\frac{1}{2}}\,[(u-v)^+]^{\frac{\beta+1}{2}}\, \varphi_R\, \mathrm{dvol}_g\\
      &\leq \tau \int_{\Omega_{2R}}\,a(z)\,[(u-v)^+]^{\beta-1}| \, \nabla
      (u-v)|^2 \, \varphi_R^2\, \mathrm{dvol}_g+\frac{1}{4\tau}
      \int_{\Omega_{2R}}\,a(z)^{-1}\,[((u-v)^+)^{\frac{\beta+1}{2}}]^2 \, \varphi_R^2\, \mathrm{dvol}_g\\
& =  \tau \mathcal{L}_{2R}\,+  \frac{1}{4\tau}
      \int_{\Omega_{2R}}\,a(z)^{-1}\,[((u-v)^+)^{\frac{\beta+1}{2}}]^2 \, \varphi_R^2\,
      \mathrm{dvol}_g\ .
   \end{split}
      \end{equation}
In order to estimate the integral on the right-hand side of the
above equality, we will consider the oriented atlas of
$\mathscr{B}\left( M, \, \e\right)$ given by
\begin{equation*}
\left\{ \, \left( \mathscr{B}(U_{\alpha}, \, \varepsilon), \,
\Phi_{\alpha}^{-1} \right) \right\},
\end{equation*}
cf. \eqref{eq:atlas-fermi}. We choose a partition of unity $ \left\{
\overline{\rho}_\alpha\right\}_\alpha$, subordinate to the open
cover $\left\{ \mathscr{B}\left( U_\alpha, \, \e\right)
\right\}_\alpha$ of $\mathscr{B}\left(M, \, \e\right)$ and obtained
by pulling back a partition of unity subordinate to the open cover
$\left\{ U_\alpha \right\}_\alpha$ of $M$ via the projection map
\begin{equation*}
\mathscr{B}(M, \, \varepsilon) \longrightarrow M, \quad x \in B(p,
\, \varepsilon) \mapsto p.
\end{equation*}
This implies that the functions $\rho_{\alpha}:=\Phi_\alpha^*\left(
\overline{\rho}_\alpha \right)$, defined on $V_{\alpha}\times
B_{\mathbb{R}^k}(0, \, \varepsilon) $, only depend on the first
factor, namely $\rho_{\alpha} = \rho_{\alpha}(x_\alpha)$. Therefore
for $\varepsilon < \varepsilon_1$ we can write, by using
\eqref{eq:two-dvol}, \eqref{eq:varepsilon_0}, \eqref{eq:dvol-h} and
the Fubini-Tonelli theorem,
\begin{equation} \label{eq:A2r-fubini}
\begin{split}
     & \int_{\Omega_{2R}}\,a(z)^{-1}\,[((u-v)^+)^{\frac{\beta+1}{2}}]^2 \,
     \varphi_R^2\, \mathrm{dvol}_g  \\
     = \sum_{\alpha}& \int_{V_{\alpha} \times B_{\mathbb{R}^k}(0,
     \, \varepsilon)} \rho_{\alpha}(x) \,  a(x, \, y)^{-1} \,
     [((u-v)^+)^{\frac{\beta+1}{2}}]^2 \, \varphi_R^2(x)\, \lambda(x, \, y) \,
     \mathrm{dvol}_{h_{\alpha}}  \\
  \leq \frac{3}{2}\sum_{\alpha}  &   \int_{V_{\alpha}} \rho_{\alpha}(x)
  \, \varphi_R^2(x)\,
   \left( \int_{B_{\mathbb{R}^k}(0, \, \varepsilon)}a(x, \, y)^{-1}
  [((u-v)^+)^{\frac{\beta+1}{2}}]^2\,dy\right)\,\mu(x)\,dx\ .
\end{split}
\end{equation}
\begin{remark}
  In order to simplify the notation, in the above equation and in the following computations
  we drop the subscript $\alpha$ whenever it is not strictly necessary. Moreover,
  with a slight abuse of notation,
  we write $a(x, \, y)$,
  $\varphi_R^2(x)$ and $(u-v)^+$
  instead of $\Phi^*_{\alpha}(a(z))$,
  $\Phi^*_{\alpha}\left(\varphi_R^2(x)\right) $ and $\Phi^*_{\alpha}\left( (u-v)^+
  \right)$, respectively, highlighting for clarity the dependence on the
  variables $x$ and $y$ if necessary.
\end{remark}

Now we estimate the integral on $B_{\mathbb{R}^k}(0, \,
\varepsilon)$ by exploiting the H\"{o}lder inequality, obtaining
\begin{equation} \label{eq:A2r-Holder}
  \begin{split}
    &
    \int_{B_{\mathbb{R}^{k}}(0,\,\e)}\,a(x, \,
    y)^{-1}\,[((u-v)^+)^{\frac{\beta+1}{2}}]^2\,dy
    \\
    \leq &
    \left(\int_{B_{\mathbb{R}^{k}}(0, \,\e)}\,\left({a(x,\,
    y)}^{-1}\right)^{\frac{kt}{2t-k}}\,dy \right)^{\frac{2t-k}{kt}}\cdot
    \left(\int_{B_{\mathbb{R}^{k}}(0,\,\e)}[((u-v)^+)^{\frac{\beta+1}{2}}]^{2^*(t)}\,dy
    \right)^{\frac{2}{2^*(t)}}.
\end{split}
\end{equation}
By exploiting again the H\"{o}lder inequality and making use of
Assumption $\mathrm{\ref{eq:h1}}$, we have
\begin{equation}
\begin{split} \label{eq:A2r-int1}
\int_{B_{\mathbb{R}^{k}}(0, \, \e)}\,\left({a(x,\,
    y)}^{-1}\right)^{\frac{kt}{2t-k}}\,dy & \leq \left(
    \int_{B_{\mathbb{R}^{k}}(0,\, \e)}\,{a(x,\,
    y)}^{-t}\,dy \right)^{\frac{k}{2t-k}} \cdot
\left( \int_{B_{\mathbb{R}^{k}}(0,\, \e)} 1 \,dy
\right)^{\frac{2t-2k}{2t-k}}
\\ & \leq
C_a^{\frac{k}{2t-k}} \left( \Gamma_k \, \e^{k}
\right)^{\frac{2t-2k}{2t-k}} \\  & =  C_a^{\frac{k}{2t-k}} \,
\Gamma_k^{\frac{2t-2k}{2t-k}} \,
\varepsilon^{\frac{k(2t-2k)}{2t-k}},
\end{split}
\end{equation}
were $\Gamma_k$ is the volume of the unit ball in $\mathbb{R}^k$.
Moreover we have, by using \eqref{Sobolev},
\begin{equation} \label{eq:A2r-int2}
\begin{split}
\left(\int_{B_{\mathbb{R}^{k}}(0,\,\e)}[((u-v)^+)^{\frac{\beta+1}{2}}]^{2^*(t)}\,dy\right)^{\frac{2}{2^*(t)}}
& = \norm{((u-v)^+)^{\frac{\beta +
1}{2}}}^2_{{L^{2^*(t)}}(B_{\mathbb{R}^k}(0,
    \, \varepsilon))} \\
    & \leq C_S^2  \norm{\nabla_{\mathbb{R}^k}
    ((u-v)^+)^{\frac{\beta + 1}{2}}}^2_{{L^{2}(B_{\mathbb{R}^k}(0, \,
    \varepsilon), \, a)}} \\
& = C_S^2 \int_{B_{\mathbb{R}^k}(0, \, \varepsilon)}
|\nabla_{\mathbb{R}^k} ((u-v)^+)^{\frac{\beta
    + 1}{2}}|^2 a(x, \,y) dy.
\end{split}
\end{equation}

\begin{remark} \label{rmk:cs}
By Theorem \ref{bvbdvvbidvldjbvlb}, see also
 \cite[proof of Theorem 5.1]{Fms2013},
  it follows that the constant $C_S$ in \eqref{eq:A2r-int2} does not depend  on $\alpha$ because of Assumption $\mathrm{\ref{eq:h1}}$.
\end{remark}

Plugging \eqref{eq:A2r-int1} and \eqref{eq:A2r-int2} into
\eqref{eq:A2r-Holder}, we infer
\begin{equation*}
\begin{split}
\int_{B_{\mathbb{R}^{k}}(0,\,\e)}\,a(x, \,
    y)^{-1}  \, & [((u-v)^+)^{\frac{\beta+1}{2}}]^2\,dy    \leq  {C}_a^{-t}
    {C}_S^2 \Gamma_k^{\frac{2t-2k}{kt}}\e^{\frac{2t-2k}{t}}
    \int_{B_{\mathbb{R}^{k}}(0, \,
    \e)}\,a(x, \,y) \,
    |\nabla_{\mathbb{R}^k}((u-v)^+)^{\frac{\beta+1}{2}}|^{2}\,dy \\
    = &  {C}_a^{-t} {C}_S^2  \Gamma_k^{\frac{2t-2k}{kt}}
    \,\e^{\frac{2t-2k}{t}} \left( \frac{\beta+1}{2} \right)^2
    \int_{B_{\mathbb{R}^{k}}(0, \, \e)}\,a(x, \,y) \,
    \left( \, [(u-v)^+]^{\frac{\beta-1}{2}}|\nabla_{\mathbb{R}^k}(u-v)| \, \right)^2\,dy
   \\
   \leq &  \frac{{C}_a^{-t} {C}_S^2
   \Gamma_k^{\frac{2t-2k}{kt}}\,\e^{\frac{2t-2k}{t}}(\beta+1)^2}{2}
   \int_{B_{\mathbb{R}^{k}}(0, \, \e)}\,a(x, \,y)\,
   [(u-v)^+]^{\beta-1}|\nabla_{\mathbb{R}^k}(u-v)|^2
   \lambda(x, \, y) \,dy,
  \end{split}
\end{equation*}
where the last inequality follows because we have $1 \leq 2
\lambda(x, \, y)$ for $\varepsilon < \varepsilon_1$, see
\eqref{eq:varepsilon_0}.

\begin{remark} \label{rmk:gradientinequality}
  By using Lemma \ref{lemma:metriccomp}, we can write
  \begin{equation*}
    |\nabla_{\Phi_\alpha^* g}(u-v)|^2 = |\nabla_{\mathbb{R}^k}(u-v)|^2  + |\nabla_{h_{\alpha}'}(u-v)|^2+ o(\varepsilon)
  \end{equation*}
   in $V_{\alpha}\times B_{\mathbb{R}^{k}}(0,\, \e)$, uniformly on $\alpha$. So, making $\varepsilon_1$ smaller if necessary, we can assume that the inequality
\begin{equation*}
|\nabla_{\mathbb{R}^k}(u-v)|^2 \leq |\nabla_{\Phi_\alpha^*
g}(u-v)|^2
\end{equation*}
holds in every coordinate chart $V_{\alpha}\times
B_{\mathbb{R}^{k}}(0,\, \e)$, as soon as $\varepsilon <
\varepsilon_1$.

Moreover, since the gradient $\nabla$ is the vector field
representing the differential map with respect to the metric (see
\cite[p. 20]{Pet98}), it is straightforward to check that
\begin{equation} \label{eq:gradient}
    \Phi_{\alpha}^*\nabla_g = \nabla_{\Phi^*_{\alpha}g}\Phi^*_{\alpha}.
\end{equation}
\end{remark}

Using Remark \ref{rmk:gradientinequality}, and noticing that all the
constants involved in the computations are independent on the local
chart $U_{\alpha}$, we obtain
\begin{equation} \label{eq:A2r-final}
  \begin{split}
    &
    \int_{\Omega_{2R}}\,a(z)^{-1}\,[((u-v)^+)^{\frac{\beta+1}{2}}]^2\varphi_R^2\,
    \mathrm{dvol}_g  \\
    \leq & \frac{3 {C}_a^{-t} {C}_S^2
    \Gamma_k^{\frac{2t-2k}{kt}}\,\e^{\frac{2t-2k}{t}}(\beta+1)^2}{4}
    \sum_{\alpha}   \int_{V_{\alpha}} \rho_{\alpha}(x)
     \Bigg( \\
    & \left. \int_{B_{\mathbb{R}^k}(0, \, \varepsilon)}
    \,a(x, \,y) \, [(u-v)^+]^{\beta-1}|\nabla_{\mathbb{R}^k}(u-v)|^2\, \varphi_R^2(x)
    \lambda(x, \, y)
    dy \right)\mu(x)\, dx \\
    \leq & \frac{3 {C}_a^{-t} {C}_S^2
    \Gamma_k^{\frac{2t-2k}{kt}}\,\e^{\frac{2t-2k}{t}}(\beta+1)^2}{4}
    \sum_{\alpha}   \int_{V_{\alpha}} \rho_{\alpha}(x)
     \Bigg( \\
    & \left. \int_{B_{\mathbb{R}^k}(0, \, \varepsilon)}
    \,a(x, \,y) \, [(u-v)^+]^{\beta-1}|\nabla_{\Phi_\alpha^* g}(u-v)|^2\, \varphi_R^2(x)
    \lambda(x, \, y)
    dy \right) \mu(x)\,dx \\
    = & \frac{3 {C}_a^{-t} {C}_S^2
    \Gamma_k^{\frac{2t-2k}{kt}}\,\e^{\frac{2t-2k}{t}}(\beta+1)^2}{4}
    \sum_{\alpha}   \int_{V_{\alpha}\times B_{\mathbb{R}^k}(0, \,
  \varepsilon)} \rho_{\alpha}(x) \, \bigg(
    \,a(x, \,y) \, [(u-v)^+]^{\beta-1}  \\
    &  |\nabla_{\Phi_\alpha^* g}(u-v)|^2\, \varphi_R^2(x) \,\bigg)\,
    \mathrm{dvol}_{\Phi_\alpha^* g} \\
    = & \frac{3 {C}_a^{-t} {C}_S^2
    \Gamma_k^{\frac{2t-2k}{kt}}\,\e^{\frac{2t-2k}{t}}(\beta+1)^2}{4}
    \int_{\Omega_{2R}}\,a(z)
    [(u-v)^+]^{\beta-1}|\nabla(u-v)|^2\varphi_R^2\, \mathrm{dvol}_g \\
    = & \frac{3 C_a^{-t} {C}_S^2
    \Gamma_k^{\frac{2t-2k}{kt}}\,\e^{\frac{2t-2k}{t}}(\beta+1)^2}{4}
    \mathcal{L}_{2R}.
  \end{split}
\end{equation}
Summing up, if $\varepsilon \leq \varepsilon_1$, from
\eqref{eq:A2r-1} and \eqref{eq:A2r-final} we finally get
\begin{equation} \label{eq:A2R}
  \mathcal{A}_{2R} \leq \left(\tau +  \frac{3 C_a^{-t} {C}_S^2
  \Gamma_k^{\frac{2t-2k}{kt}}\,\e^{\frac{2t-2k}{t}}(\beta+1)^2}{16 \tau}
  \right) \mathcal{L}_{2R}.
\end{equation}

\subsection{Estimate for $\mathcal{B}_{2R}$}
Let us now estimate the term $\mathcal{B}_{2R}$. As before, for
$\varepsilon \leq \varepsilon_1$ we have
\begin{equation*}
  \begin{split}
    \mathcal{B}_{2R} & = \int_{ \Omega_{2R}} [(u-v)^{+}]^{\beta+1}
    {\varphi_R}^2 \mathrm{dvol}_g  \\
    & = \sum_{\alpha} \int_{V_{\alpha} \times B_{\mathbb{R}^k}(0, \,
    \varepsilon)}
    \rho_{\alpha}(x) [(u-v)^{+}]^{\beta+1} {\varphi_R}^2(x) \lambda(x, \,
    y) \, \mathrm{dvol}_{h_{\alpha}} \\
    \leq & \frac{3}{2}  \sum_{\alpha} \int_{V_{\alpha}} \rho_{\alpha}(x)
    {\varphi_R}^2(x)  \left( \int_{B_{\mathbb{R}^k}(0, \, \varepsilon)}
    [(u-v)^{+}]^{\beta+1} dy \right) \mu(x) \,dx.
  \end{split}
\end{equation*}
Now we estimate the integral on $B_{\mathbb{R}^k}(0, \,
\varepsilon)$ by using the H\"{o}lder inequality and
\eqref{eq:A2r-int2}, obtaining
\begin{equation} \label{eq:estimateB}
  \begin{split}
    \int_{B_{\mathbb{R}^k}(0, \, \varepsilon)} & [(u-v)^{+}]^{\beta+1} \,
    dy \leq
    \left(\int_{B_{\mathbb{R}^{k}}(0, \,
    \e)}\,1\,\,dy\right)^{\frac{2t-k}{kt}}\cdot
    \left(\int_{B_{\mathbb{R}^{k}}(0, \, \e)}
    [((u-v)^+)^{\frac{\beta+1}{2}}]^{2^*(t)}\,dy \right)^{\frac{2}{2^*(t)}} \\
    & \leq \left(\Gamma_k  \, \varepsilon^k \right)^{\frac{2t-k}{kt}}C_S^2
    \int_{B_{\mathbb{R}^k}(0, \, \varepsilon)} |\nabla_{\mathbb{R}^k}
    ((u-v)^+)^{\frac{\beta
    + 1}{2}}|^2 \, a(x, \, y) dy  \\
    & = \Gamma_k^{\frac{2t-k}{kt}} \, \varepsilon^{\frac{2t-k}{t}}C_S^2
    \left(\frac{\beta+1}{2} \right)^2
    \int_{B_{\mathbb{R}^k}(0, \, \varepsilon)} \big( \,[(u-v)^+]
    ^{\frac{\beta-1}{2}} \, |\nabla_{\mathbb{R}^k}(u-v)| \, \big)^2 \, a(x, \,
    y) \, dy  \\
    & = \frac{\Gamma_k^{\frac{2t-k}{kt}} \, \varepsilon^{\frac{2t-k}{t}}C_S^2
    \,
    (\beta+1)^2}{4}
    \int_{B_{\mathbb{R}^k}(0, \, \varepsilon)} [(u-v)^+]^{\beta-1}
    |\nabla_{\mathbb{R}^k}(u-v)|^2\, a(x, \, y) \,dy \\
    & \leq \frac{\Gamma_k^{\frac{2t-k}{kt}} \,
    \varepsilon^{\frac{2t-k}{t}}C_S^2 \,
    (\beta+1)^2}{2}
    \int_{B_{\mathbb{R}^k}(0, \, \varepsilon)} [(u-v)^+]^{\beta-1}
    |\nabla_{\mathbb{R}^k}(u-v)|^2 \, a(x, \, y) \, \lambda(x, \, y) \, dy.
  \end{split}
\end{equation}
Arguing as in \eqref{eq:A2r-final} and using Remark
\ref{rmk:gradientinequality}, we deduce
\begin{equation} \label{eq:B2r-final}
  \begin{split}
    \mathcal{B}_{2R} & \leq \frac{3 \, \Gamma_k^{\frac{2t-k}{kt}} \,
    \varepsilon^{\frac{2t-k}{t}}C_S^2 \, (\beta+1)^2}{4}  \sum_{\alpha}
    \int_{V_{\alpha}} \rho_{\alpha}(x) \,
     \Bigg( \\
    & \left. \int_{B_{\mathbb{R}^k}(0, \, \varepsilon)}
    \,a(x, \, y)\, [(u-v)^+]^{\beta-1} \, |\nabla_{\mathbb{R}^k}(u-v)|^2\,
    \lambda(x, \, y) \, \varphi_R^2(x)
   \, dy \right)  \mu(x)\, dx \\
    & \le \frac{3 \, \Gamma_k^{\frac{2t-k}{kt}} \,
    \varepsilon^{\frac{2t-k}{t}}C_S^2 \,
    (\beta+1)^2}{4} \sum_{\alpha}   \int_{V_{\alpha}\times B_{\mathbb{R}^k}(0, \,
  \varepsilon)} \rho_{\alpha}(x)\,  \bigg(
    \,a(x, \,y) \, [(u-v)^+]^{\beta-1} \\
    &  |\nabla_{\Phi_\alpha^* g}(u-v)|^2\, \varphi_R^2(x) \, \bigg) \,
    \mathrm{dvol}_{\Phi_\alpha^* g} \\
    & = \frac{3 \, \Gamma_k^{\frac{2t-k}{kt}} \,
    \varepsilon^{\frac{2t-k}{t}}C_S^2 \,
    (\beta+1)^2}{4} \int_{\Omega_{2R}}\,a(z) \,
    [(u-v)^+]^{\beta-1} \, |\nabla(u-v)|^2 \,\varphi_R^2 \, \mathrm{dvol}_g.
  \end{split}
\end{equation}
Summing up, for $\varepsilon < \varepsilon_1$ we get
\begin{equation} \label{eq:B2R}
\mathcal{B}_{2R} \leq \frac{3 C_S^2 \, \Gamma_k^{\frac{2t-k}{kt}} \,
\varepsilon^{\frac{2t-k}{t}} \, (\beta+1)^2}{4} \mathcal{L}_{2R}.
\end{equation}

\subsection{Estimate for $\mathcal{C}_{2R}$}
Finally, let us estimate the term $\mathcal{C}_{2R}$. Exploiting
the Young inequality and the last inequality in \eqref{Eq:Cut-off1},
we get
\begin{equation*}
  \begin{split}
    \mathcal{C}_{2R} & =  \int_{\Omega_{2R}} a(z) \, [(u-v)^+]^{\beta}\,
    |\nabla(u-v)| \,|\nabla\varphi_R|\,  \varphi_R \, \mathrm{dvol}_g \\
     & = \int_{\Omega_{2R}}\,a(z)^{\frac{1}{2}}\, [(u-v)^+]^{\frac{\beta-1}{2}}\, |\nabla
    (u-v)| \, \varphi_R \, a(z)^{\frac{1}{2}}\,[(u-v)^+]^{\frac{\beta+1}{2}}\, |\nabla\varphi_R|\, \mathrm{dvol}_g\\
    & \leq \tau'   \int_{\Omega_{2R}}\,a(z)\,[(u-v)^+]^{\beta-1}\, |\nabla
    (u-v)|^2\,\varphi_R^2\,\mathrm{dvol}_g  \\
    &  + \frac{1}{4 \tau'}
    \int_{\Omega_{2R}}\,a(z) \, [(u-v)^+]^{\beta+1}\,|\nabla\varphi_R|^2 \, \mathrm{dvol}_g\\
& \leq \tau' \mathcal{L}_{2R} + \frac{\norm{a}_{\infty}}{\tau' R^2}
\int_{\Omega_{2R}} [(u-v)^+]^{\beta+1}\,\mathrm{dvol}_g.
\end{split}
\end{equation*}
If $\varepsilon < \varepsilon_1$ the integral at the right-hand side
can be estimated by \eqref{eq:estimateB}, obtaining
\begin{equation} \label{eq:C2r-final}
  \begin{split}
    \int_{  \Omega_{2R}}& [(u-v)^{+}]^{\beta+1} \, \mathrm{dvol}_g  =
    \sum_{\alpha} \int_{V_{\alpha} \times B_{\mathbb{R}^k}(0, \, \varepsilon)}
    \rho_{\alpha}(x)\, [(u-v)^{+}]^{\beta+1} \, \lambda(x, \, y)
    \, \mathrm{dvol}_{h_{\alpha}} \\
    & \leq  \frac{3}{2}   \sum_{\alpha} \int_{V_{\alpha}} \rho_{\alpha}(x)
    \,     \left(
    \int_{B_{\mathbb{R}^k}(0, \, \varepsilon)}
    [(u-v)^{+}]^{\beta+1} \,dy \right) \mu(x)\, dx \\
    & \leq \frac{ 3 \Gamma_k^{\frac{2t-k}{kt}}
    \varepsilon^{\frac{2t-k}{t}} C_S^2
    (\beta+1)^2}{4} \sum_{\alpha} \int_{V_{\alpha}} \rho_{\alpha}(x)
    \,  \Bigg( \\
    & \left. \int_{B_{\mathbb{R}^k}(0, \, \varepsilon)}
    \,a(x, \,y) \, [(u-v)^+]^{\beta-1} \, |\nabla_{\mathbb{R}^k}(u-v)|^2\,
    \lambda(x, \, y) \,
    dy \right) \mu(x) \, dx \\
    & \le \frac{3 \, \Gamma_k^{\frac{2t-k}{kt}} \,
    \varepsilon^{\frac{2t-k}{t}}C_S^2 \,
    (\beta+1)^2}{4}
    \sum_{\alpha}   \int_{V_{\alpha}\times B_{\mathbb{R}^k}(0, \,
    \varepsilon)} \rho_{\alpha}(x) \, \bigg(
    \,a(x, \,y) \, [(u-v)^+]^{\beta-1}  \\
    &  |\nabla_{\Phi_\alpha^* g}(u-v)|^2\, \varphi_R^2(x) \, \bigg)
    \, \mathrm{dvol}_{\Phi_\alpha^* g} \\
    & = \frac{3 \, \Gamma_k^{\frac{2t-k}{kt}} \,
    \varepsilon^{\frac{2t-k}{t}}C_S^2 \,
    (\beta+1)^2}{4} \int_{\Omega_{2R}}\,a(z) \,
    [(u-v)^+]^{\beta-1} \, |\nabla(u-v)|^2\, \mathrm{dvol}_g.
  \end{split}
\end{equation}
Summing up, for $\varepsilon < \varepsilon_1$ we get
\begin{equation} \label{eq:C2R}
\mathcal{C}_{2R} \leq \tau' \mathcal{L}_{2R} + \frac {3 \, C_S^2\,
\Gamma_k^{\frac{2t-k}{kt}} \, \varepsilon^{\frac{2t-k}{t}} \,
(\beta+1)^2 \, \norm{a}_{\infty}}{4 \tau' R^2}
\int_{\Omega_{2R}}\,a(z) \, [(u-v)^+]^{\beta-1} \, |\nabla(u-v)|^2\,
\mathrm{dvol}_g.
\end{equation}

\subsection{End of the proof} \label{subsec:end}
We can now specialize the computations above, using particular
values for the parameters. Fixing, for instance,
\begin{equation*}
  \beta = 2, \qquad \tau'=\frac{1}{4},
\end{equation*}
we can rewrite inequalities \eqref{eq:A2R}, \eqref{eq:B2R},
\eqref{eq:C2R} as follows:
\begin{equation*}
\begin{split}
\mathcal{A}_{2R} & \leq \left( \tau +
\frac{\Theta_{\mathcal{A}}}{\tau} \right) \mathcal{L}_{2R}, \\
\mathcal{B}_{2R} &
\leq \Theta_{\mathcal{B}}\, \mathcal{L}_{2R}, \\
\mathcal{C}_{2R} & \leq  \frac{1}{4} \, \mathcal{L}_{2R} +
\frac{\Theta_{\mathcal{C}}}{R^2} \int_{\Omega_{2R}}\,a(z) \, (u-v)^+
\, |\nabla(u-v)|^2\, \mathrm{dvol}_g,
\end{split}
\end{equation*}
where
\begin{equation} \label{eq:wTheta}
\begin{split}
\Theta_{\mathcal{A}} & = \frac{ 27 \, C_a^{-t} {C}_S^2
  \Gamma_k^{\frac{2t-2k}{kt}}\,\e^{\frac{2t-2k}{t}}}{16} \\
\Theta_{\mathcal{B}} & = \frac{27 \, C_S^2 \,
\Gamma_k^{\frac{2t-k}{kt}}
\, \varepsilon^{\frac{2t-k}{t}}}{4} \\
\Theta_{\mathcal{C}} & = 27 \, C_S^2\, \Gamma_k^{\frac{2t-k}{kt}} \,
\varepsilon^{\frac{2t-k}{t}} \,
 \norm{a}_{\infty},
\end{split}
\end{equation}
so that \eqref{eq:ABCL} becomes
\begin{equation*} \label{eq:ABCL-1}
\begin{split}
2 \mathcal{L}_{2R} & \leq |\Lambda| \, q \; (\norm{\nabla u}_{\infty} + \norm{\nabla v}_{\infty})^{q-1} \left( \tau +\frac{\Theta_{\mathcal{A}}}{\tau} \right) \, \mathcal{L}_{2R} +  L_f  \, \Theta_{\mathcal{B}}\, \mathcal{L}_{2R} \\
 & +  \frac{1}{2} \, \mathcal{L}_{2R} + \frac{2 \, \Theta_{\mathcal{C}}}{R^2}  \int_{\Omega_{2R}}\,a(z) \,
(u-v)^+ \, |\nabla(u-v)|^2\, \mathrm{dvol}_g.
\end{split}
\end{equation*}
Finally, we fix
\begin{equation*}
\tau = \frac{1}{2\,|\Lambda|\, q\; (\|\nabla u\|_\infty + \|\nabla
  v\|_\infty )^{q-1}}
\end{equation*}
and we resume the previous computations as
  \begin{equation} \label{eq:L-Theta1-Theta2}
\mathcal{L}_{2R} \leq \Theta_1 \, \mathcal{L}_{2R} +
\frac{\Theta_2}{R^2}  \int_{\Omega_{2R}}\,a(z) \,(u-v)^+ \,
|\nabla(u-v)|^2\, \mathrm{dvol}_g,
  \end{equation}
where
\begin{equation} \label{eq:Theta1-Theta2}
\Theta_1:= \left( 2 \, |\Lambda|^2 \, q^2 \; (\norm{\nabla
 u}_{\infty} + \norm{\nabla v}_{\infty})^{2q-2}
\, \Theta_{\mathcal{A}} + L_f \Theta_{\mathcal{B}} \right), \qquad
\Theta_2 := 2 \Theta_{\mathcal{C}}.
\end{equation}
Let us now set
  \begin{equation*}
    \widetilde{\mathcal{
    L}}_{R}\,:=\int_{\Omega_{R}}\,a(z) \,(u-v)^+ \, |\nabla(u-v)|^2\, \mathrm{dvol}_g,
\end{equation*}
    so that, from the properties of $\varphi_R$ stated in \eqref{Eq:Cut-off1}, it follows
    \begin{equation*}
      \widetilde{\mathcal{L}}_{R} \leq \mathcal{L}_{2R} \le
      \widetilde{\mathcal{L}}_{2R}.
    \end{equation*}
    Therefore by \eqref{eq:L-Theta1-Theta2} we get
    \begin{equation} \label{eq:theta-final}
      \widetilde{\mathcal{L}}_{R} \leq \theta \widetilde{\mathcal{L}}_{2R},
    \end{equation}
where
\begin{equation*}
\theta:= \left( \Theta_1 + \frac{\Theta_2}{R^2} \right).
\end{equation*}
Looking at \eqref{eq:wTheta} and \eqref{eq:Theta1-Theta2}  we see
that, taking  $\e  < \e_1$ sufficiently small and $R$ sufficiently
large, we have $\theta<2^{-\gamma}$, where $\gamma$ is as in
Assumption $\mathrm{\ref{eq:h4}}$.
    Moreover, we are supposing $a, \, u, \, v, \, \nabla u, \, \nabla v \in L^\infty\left( \Omega \right)$ so, using Assumption $\mathrm{\ref{eq:h4}}$, we obtain
    \begin{equation*}
      \widetilde{\mathcal{L}}_{R} \leq C R^{\gamma}\ ,
    \end{equation*}
    where we set $C:=C_1 \, \|a(z) \,(u-v)^+ \, |\nabla(u-v)|^2\|_{L^\infty\left( \Omega \right)}$.

Summing up, there exists $0 < \varepsilon_0 < \varepsilon_1$  such
that, if $\varepsilon < \varepsilon_0$,   we can apply  Lemma
\ref{Le:L(R)} in order to deduce $\mathcal{\widetilde{L}}_{R}=0$ for
all $R >0$. This in turn implies $(u-v)^+ \, |\nabla(u-v)^+|^2=0$ in
the whole of $\Omega$. Since $(u-v)^+=0$ on $\partial \Omega$, we
get $(u-v)^+=0$ everywhere, that is $u \leq v$ on the whole of
$\Omega$ and the weak comparison principle is proved.

\bigskip

\section*{Acknowledgements} We are grateful to the MathOverflow users
Jaap Eldering, Ivan Izmestiev, macbeth, Anton Petrunin, Raziel, Rbega, Deane Yang, Fan Zheng for interesting discussions in the threads \\ \verb|https://mathoverflow.net/questions/283467|, \\ \verb|https://mathoverflow.net/questions/284388| \\ \verb|https://mathoverflow.net/questions/285061|.  \\
We also thank Venere Fusaro for her help with Figure \ref{fig:omega}. \\

\bigskip \bigskip

    \end{document}